\documentclass{article}
\usepackage{stmaryrd}

\usepackage[utf8]{inputenc}
\usepackage[T1]{fontenc}
\usepackage[british,UKenglish,USenglish,english,american]{babel}

\usepackage{hyperref}
\usepackage{amsmath}
\usepackage{amssymb}
\usepackage{amsthm}
\usepackage{mathrsfs}
\usepackage{amsfonts}

\usepackage{esint}
\usepackage{array}
\usepackage{latexsym}

\usepackage{lmodern}

\usepackage{graphicx}
\usepackage{graphics}
\usepackage{epsfig}
\usepackage{color}
\DeclareGraphicsExtensions{.eps,.pdf,.jpeg,.png}

\usepackage{titlesec}
\usepackage{multirow}

\usepackage[all]{xy}
\usepackage{dsfont}

\usepackage{hyperref}
\usepackage{bigints}

\def\Ch{\operatorname{Ch}}

\def\ind{\operatorname{ind}}

\def\Tr{\operatorname{Tr}}

{\theoremstyle {definition} \newtheorem {defi} {Definition} [section] }
{\theoremstyle {plain}  \newtheorem {thm} [defi] {Theorem}}
{\theoremstyle {plain}  \newtheorem {cor} [defi]{Corollary}}
{\theoremstyle {plain} \newtheorem {prop} [defi]{Proposition}}
{\theoremstyle {plain} \newtheorem {lem}[defi] {Lemma}}

\newtheorem{notation}{Notation}[defi]
{\theoremstyle {definition} }

{\theoremstyle {definition} \newtheorem{remarque}[defi]{Remark}}
{\theoremstyle {definition} }
{\theoremstyle {definition} }
{\theoremstyle {definition}  }
{\theoremstyle {definition} }

\def\Eul{{\mathrm{Eul}}}

\def\Str{{\mathrm{Str}}}
\def\Tr{{\mathrm{Tr}}}

\def\Ch{{\mathrm{Ch}}}
\def\det{{\mathrm{det}}}
\def\exp{{\mathrm{exp}}}
\def\g{{\mathfrak{g}}}

\def\K{{\mathrm{K_G}}}

\def\Khh{{\mathrm{K_{H}}}}
\def\k{{\mathrm{K}}}
\def\ind{{\mathrm{Ind^{M|B}}}}

\usepackage{authblk}
\newcommand{\email}[1]{\href{mailto:#1}{#1}}

\title{The index of $G$-transversally elliptic families II}
\author[1]{Alexandre Baldare
}
\affil[1]{Institut \'Elie Cartan de Lorraine, Université de Lorraine, 57070 Metz, France, \email{alexandre.baldare@univ-lorraine.fr}}

\begin{document}

\maketitle

\begin{abstract}
We define the Chern character of the index class of a $G$-invariant family of $G$-transversally elliptic operators, see \cite{baldare:KK}. Next we study the Berline-Vergne formula for families in the elliptic and transversally elliptic case.
\end{abstract}

\tableofcontents

\section{Introduction}
\quad
Let $M$ be a compact manifold. Let $G$ be a compact Lie group acting on $M$. Assume for simplicity that $G$ is topologicaly cyclic and generated by an element $g\in G$. Let $i : M^g \hookrightarrow M$ denote the inclusion of the fixed point submanifold of $g$. In \cite{Atiyah-Segal:II}, Atiyah and Segal gave a Lefschetz fixed point formula for the index of a $G$-invariant elliptic operator on $M$ using localization in equivariant $\k$-theory. 
Denote by $R(G)_g$ the localization of the ring $R(G)$ at the prime ideal $I_g=\{\chi \in R(G) ,\chi(g)=0\}$. Atiyah and Segal obtained 
$$\mathrm{Ind}^{\mathrm{M}}_\mathrm{G}(\sigma)=\bigg(\mathrm{Ind}^{\mathrm{M^g}}\otimes id_{R(G)_g}\bigg)\bigg(\dfrac{i^*\sigma}{\lambda_{-1}(N^g \otimes \mathbb{C})}\bigg),$$
where $\lambda_{-1}(N^g \otimes \mathbb{C}) = \sum (-1)^i \bigwedge^i(N^g \otimes \mathbb{C})$ and $\sigma \in \K(T^*M)$ is the homotopy class of an elliptic symbol. 
In \cite{Atiyah-Singer:III}, Atiyah and Singer obtained, using the Lefschetz fixed point formula, a cohomological version of the Lefschetz fixed point formula. More precisely, Atiyah and Singer obtained the following formula
$$\mathrm{Ind}^{\mathrm{M}}_\mathrm{G}(\sigma)(g)=(-1)^{n_g} \bigint_{T^*M^g}\dfrac{\Ch(i^*\sigma (g)) }{\Ch(\lambda_{-1}(N^g \otimes \mathbb{C})(g))}\wedge \hat{A}(T^*M^g)^2,$$
where $i : T^*M^g \hookrightarrow T^*M$ is the inclusion of the submanifold of fixed points using a riemannian metric and $i^*\sigma(g)$ means the evaluation of $i^*\sigma \in \k(T^*M^g) \otimes R(G)_g$ at the element $g$, $\Ch$ means the Chern character tensored by the identity of $\mathbb{C}$ and $\hat{A}(T^*M^g)$ is the $\hat{A}$-genus of $M^g$.\\
In \cite{BV:formuleloc:Kirillov}, Berline and Vergne gave a delocalized formula for the index of a $G$-invariant elliptic operator in equivariant cohomology. To obtain this new formula, Berline and Vergne showed a localization formula in equivariant cohomology and then applied the localization formula to derive a delocalized cohomological Lefschetz theorem. More precisely, Berline and Vergne obtained the following formula 
$$\mathrm{Ind^{M}_G}(\sigma)(se^X)=\bigint_{TM^s}\dfrac{\Ch_{s}(\sigma,X)\wedge \hat{A}^2(T^VM^s,X)}{D_s(N_s,X)},$$
where $s\in G$, $X$ is an element of the Lie algebra $\g(s)$ of the centralizer $G(s)$ of $s$ and $\Ch_{s}(\sigma,X)$, $\hat{A}^2(T^VM^s,X)$ and $D_s(N_s,X)$ are equivariant characteristic cohomology classes (see Section \ref{section:equiv:form} and \cite{BGV,BV:formuleloc:Kirillov,BV:equiChernCaracter} for more details). Then in \cite{BV:equiChernCaracter}, Berline and Vergne gave an index theorem in equivariant cohomology for $G$-invariant elliptic operators. \\

For $G$-transversally elliptic operators, Atiyah and Singer defined in \cite{atiyah1974elliptic} an index class which is now an invariant slowly increasing distribution on $G$. Similar to the case of elliptic operators, Berline and Vergne showed a corresponding index theorem for $G$-transversally elliptic operators, see \cite{BV:IndEquiTransversal}. In \cite{paradan2008index}, Paradan and Vergne gave a new approach to the cohomological index theorem for $G$-transversally elliptic operators using the Chern character of \cite{paradan2008equivariant} but with generalized coefficients on the Lie algebra $\g$ of $G$.\\

In the present article, we investigate the cohomological form of the index of families of $G$-transversally elliptic operators. Let $p : M \rightarrow B$ be a $G$-equivariant compact fibration. Assume that $G$ acts trivially on $B$. In \cite{baldare:KK}, an index class for $G$-invariant families of $G$-transversally elliptic operators was defined in the Kasparov bivariant $\k$-theory group $\k\k(C^*G,C(B))$ and using results from \cite{hilsum1987morphismes} it was then shown that the Kasparov product of this index class with an elliptic operator on the base $B$ is given by the index class of a $G$-transversally elliptic operator on $M$. Using this fact and the well defined bivariant Chern character in local cyclic homology \cite{puschnigg2003diffeotopy}, we introduce the Chern character of our index class and we show, as expected, that it belongs to the space of $G$-invariant slowly increasing distribution on $G$ with values in the de Rham cohomology of $B$. \\
Once this Chern character is defined, we investigate the Berline-Vergne formula in this context. As in the case of a single operator, we start by the simpler case of $G$-invariant families of elliptic operators. In this case, we deduce a Berline-Vergne formula using the Lefschetz formula for families, see \cite{Benameur:thmFamilleLefschetz}. In \cite{Benameur:thmFamilleLefschetz}, the Lefschetz formula  for famillies is obtained using results from \cite{Benameur:LongLefschetzKtheorie}, where a longitudinal Lefschetz theorem is shown in $\k$-theory. For foliations, Lefschetz formulas can be found in \cite{Benameur_Heitsch_lefschetz_foliation, heitsch1990lefschetz}. Note that our Berligne-Vergne formula for $G$-invariant families of elliptic operators is also valid when $G$ acts non trivially on $B$.  More precisely, if $s\in G$ and $\sigma$ is a $G$-invariant elliptic symbol along the fibers then we obtain the following theorem:

\begin{thm}\label{thm:BV:famille:elliptique} Let $s\in G$ and $X\in \mathfrak{g}(s)$. Denote by $N_s$ the normal bundle of $M^s$ in $M$.
The following equality is true in the cohomology $H(B^s,d_X)$ : 
$$\Ch_{s}\big(\mathrm{Ind^{M|B}_H}(\sigma),X\big)=\int_{T^VM^s|B^s}\dfrac{\Ch_{s}(\sigma,X)\wedge \hat{A}^2(T^VM^s,X)}{D_s(N_s\cap T^VM,X)}.$$ 
where $B^s$ the fixed point submanifold in the base $B$ and $T^VM^s$ is the vertical tangent bundle to the fixed point submanifold $M^s$.
\end{thm}

For families of $G$-transversally elliptic operators, we have to assume that the action of $G$ on $B$ is trivial and we use the Paradan-Vergne approach of equivariant cohomology with generalized coefficient. In this case we obtain the following theorem: 

\begin{thm}\label{thm:BV:familles}
Let $\sigma$ be a $G$-transversally elliptic symbol along the fibers of a compact $G$-equivariant fibration $p : M\rightarrow B $ with $B$ oriented and $G$-trivial. Denote by $N_s$ the normal vector bundle to $M^s$ in $M$.\\
1. There is a unique generalized function with values in the cohomology of $B$ denoted $\mathrm{Ind}^{G,M|B}_{coh} : \K(T^V_GM) \rightarrow C^{-\infty}(G , H(B,\mathbb{C}))^{Ad(G)}$ satisfying the following local relations:
$$\mathrm{Ind}^{G,M|B}_{coh}([\sigma ] )\|_s(Y)=(2i\pi)^{-\dim (M^s|B)}\bigint_{T^VM^s|B}\hspace*{-0.5cm}\dfrac{\Ch_s(\mathbb{A}^{r^*\omega_s}(\sigma),Y)\wedge \hat{A}^2(T^VM^s,Y)}{D_s(N^s,Y)},$$
$\forall s\in G$, $\forall Y \in \g(s)$ small enough so that the equivariant classes $\hat{A}^2(T^VM^s,Y)$ and $D(N^s,Y)$ are defined. \\
2. Furthermore, we have the following index formula:
$$\mathrm{Ind}^{G,M|B}_{coh}([\sigma ] )=\Ch(\ind([\sigma ]))\in C^{-\infty}(G ,H(B,\mathbb{C}))^{Ad(G)}.$$
Here $\Ch$ is the Chern character, defined using bivariant local cyclic homology from \cite{puschnigg2003diffeotopy}.
\end{thm}

\medskip

Let us describe the contents of the present paper. The second section is devoted to some preliminary results. We recall the definition of the equivariant cohomology which will be used later on, then we review the definition of the index class of a family of $G$-transversally elliptic operators from \cite{baldare:KK}, and we finish by a overview of the Atiyah-Singer index theorem for equivariant families \cite{Atiyah-Singer:IV}. In the third section, we define the Chern character of the index class of a $G$-invariant family of $G$-transversally elliptic operators, an invariant distribution with values in the de Rham cohomology of the base $B$. In the fourth and fifth sections, we prove the Berline-Vergne formulae for families of elliptic and $G$-transversally elliptic operators.

\medskip
\noindent
\textbf{Acknowledgements.}
This work is part of my PhD thesis under the supervision of M.-T. Benameur. I would like to thank my advisor for very helpful discussions, comments and corrections. 
I also thank W. Liu, P.-E. Paradan and V. Zenobi for several conversations during the preparation of this work. I am also indebted to M. Hilsum, P. Piazza, M. Puschnigg and G. Skandalis for reading the PhD version of this work and for their constructive suggestions. I would also like to thank
the referee for several very useful suggestions. Last but not least, I would like to thank P. Carrillo Rouse for his interest, insight, and useful discussions during the \textit{Workshop on Index Theory, Interactions and Applications} in Toulouse.


\section{Some preliminary results}\label{III}
Let $G$ be a compact group.
Let $p:M\rightarrow B$ be a $G$-equivariant (locally trivial) fibration of compact manifolds, with typical fiber $F$, a compact manifold. We denote by $M_b=p^{-1}(b)$ the fiber over $b \in B$, by $T^VM = \textrm{ker}~ p_*$ the vertical subbundle of $TM$, and by $T^VM^*$ its dual bundle. We choose a $G$-invariant riemannian metric on $M$ and hence will identify $T^VM^*$ with a subbundle of $T^*M$ when needed. 
\noindent
Let $\pi : E=E^+ \oplus E^- \rightarrow M$ be a $\mathbb{Z}_2$-graded vector bundle on $M$ which is assumed to be $G$-equivariant with fixed $G$-invariant hermitian structure. Denote by $\mathcal{P}^m(M,E^+,E^-)$ the space of (classical) continuous families of pseudodifferential operators on $M$ as defined in \cite{Atiyah-Singer:IV}. A family $P$ is $G$-equivariant if $g\cdot P=g\circ P \circ g^{-1}=P$, for any $ g\in G$.  As usual, $C^{\infty,0}(M,E)$ will be the space of continuous fiberwise smooth sections of $E$ over $M$, see \cite{Atiyah-Singer:IV}. 

\subsection{Equivariant cohomology}\label{section:equiv:form}
Here we recall the definition of equivariant cohomologies and equivariant forms which will be used in the sequel, see \cite{BGV,BV:equiChernCaracter,BV:formuleloc:Kirillov}. Let $L$ be a compact Lie group and $\mathfrak{l}$ its Lie algebra. Assume that $L$ acts on a manifold $W$ (we say that $W$ is a $L$-manifold). Let $X\in \mathfrak{l}$. Denote by $X^*_W$ the vector field generated by $X$ on $W$ that is $X^*_W(f)(w)=\dfrac{d}{dt}_{|t=0}f(e^{-tX}\cdot w)$, $\forall f\in C^{\infty}(W),\ w\in W$. Let $d$ be the de Rham differential and let $\iota (Y) $ denote the contraction by a vector field $Y$. 
Let $\mathcal{A}(W)$ be the space of differential form on $W$. Denote by $\mathcal{A}(W)_X$ the subspace of $\mathcal{A}(W)$ given by the forms $\alpha$ such that $\mathscr{L}(X)\alpha =0$. Let $d_X$  denote the operator $d-\iota (X_W^*)$ on $\mathcal{A}(W)$.

\begin{defi}
The $d_X$-cohomology $H(W,d_X)$ of $W$ is the cohomology of the complex $(\mathcal{A}(W)_X,d_X)$.
\end{defi}
Denote by $W^X=\{w\in W,\  X_W^*(w)=0\}$ and by $i : W^X \hookrightarrow W$ the inclusion. Since $X_W^*=0$ on $W^X$, we get that $d_X$ is the usual de Rham differential on $W^X$ and so $H(W^X,d_X)$ coincides with the usual de Rham cohomology $H(W^X,\mathbb{C})$ of $W^X$. Recall from \cite{BV:formuleloc:Kirillov} that $i^* : H(W,d_X) \rightarrow H(W^X,\mathbb{C})$ is an isomorphism. 

Let $\mathcal{A}^{\infty}_L(\mathfrak{l},W)$ denote the algebra $\big(C^{\infty}(\mathfrak{l}) \otimes \mathcal{A}(W)\big)^L$ of $L$-invariant smooth functions on $\mathfrak{l}$ with values in $\mathcal{A}(W)$.
Let $d_\mathfrak{l}$ be the operator on $\mathcal{A}^{\infty}_L(\mathfrak{l},W)$ given by 
$$(d_\mathfrak{l}\alpha)(X)=d(\alpha(X))-\iota (X^*_W)(\alpha(X)).$$
We have $(d_\mathfrak{l}^2\alpha)(X)=-\mathscr{L}(X)\alpha(X)$ so $d_\mathfrak{l}^2$ is zero on $\mathcal{A}^{\infty}_L(\mathfrak{l},W)$ because any element of $\mathcal{A}^{\infty}_L(\mathfrak{l},W)$ is $L$-invariant.
\begin{defi}

The equivariant cohomology $\mathcal{H}_L^{\infty}(\mathfrak{l} ,W)$ with smooth coefficients is the cohomology of the complex $(\mathcal{A}^{\infty}_L(\mathfrak{l} ,W),d_{\mathfrak{l}})$.
\end{defi}

Let $E$ be a $\mathbb{Z}/2\mathbb{Z}$-graded vector bundle on $W$ . Let $\mathcal{A}^{\infty}_L(W,E)=\big(C^\infty (\mathfrak{l})\otimes \mathcal{A}(W,E)\big)^L$, where $\mathcal{A}(W,E)$ denote the differential form on $W$, with values in $E$. 
\noindent
Let us recall the definitions of a super-connection and its curvature \cite{Quillen:superco,Quillen:superco:thom,BGV}.  
\begin{defi}
A super-connection on $E$ is an odd-parity first-order differential operator
$$\mathbb{A} : \mathcal{A}^{\pm}(M,E) \rightarrow \mathcal{A}^{\mp}(M,E),$$
which satisfies Leibniz's rule in the $\mathbb{Z}/2\mathbb{Z}$-graded sense: if $\alpha \in \mathcal{A}(M)$ and $\theta \in \mathcal{A}(M,E)$, then
$$\mathbb{A}(\alpha \wedge \theta)=d\alpha \wedge \theta +(-1)^{|\alpha|}\alpha \wedge \mathbb{A}\theta,$$
where $|\alpha|$ is the degree of $\alpha$. \\
The curvature $F$ of a super-connection $\mathbb{A}$ is the operator $F=\mathbb{A}^2$ on $\mathcal{A}(M,E)$.
\end{defi}

Let us recall the definition of the usual Chern character of a super-connexion. 
\begin{defi}
Let $\mathbb{A}$ be a super-connection on $E$. The Chern character of $\mathbb{A}$ is defined by:
$$\Ch(\mathbb{A})=\Str(e^F) \in \mathcal{A}(W),$$
where $\Str$ is the super-trace on $E$.
\end{defi}

\noindent
Let $\mathbb{A}$ be a $L$-invariant super-connection on $E$. The operator $\mathbb{A}_\mathfrak{l}$ defined by
$$(\mathbb{A}_\mathfrak{l}\alpha)(X)=(\mathbb{A}-\iota(X_W^*))(\alpha(X)), ~\forall \alpha \in \mathcal{A}(W,E) ~\mathrm{and}~X\in \mathfrak{l},$$ is called the equivariant super-connection. The equivariant curvature $F_\mathfrak{l}$ is given by $F_\mathfrak{l}(X)=(\mathbb{A}-\iota(X))^2+\mathscr{L}^E(X)$. Denote by $\mu(X)=\mathscr{L}(X)-[\iota(X),\mathbb{A}]$ the moment of $X\in \mathfrak{l}$ with respect to the super-connection $\mathbb{A}$. Then we have $F_\mathfrak{l}(X)=F+\mu(X)$, where $F$ is the curvature of $\mathbb{A}$.\\
Let us recall some equivariant forms. Denote by $W^s$ the submanifold of fixed points $\{w\in W | s\cdot w=w\}$. Denote by $L(s)$ the subgroup of $L$ of those elements that commute with $s$. We denote by $\mathfrak{l}(s)$ the Lie algebra of $L(s)$, it consists of the elements $X \in \mathfrak{l}$ such that $s\cdot X=X$. On the submanifold $W^s$ of fixed points of $s$, the action of $s$ on $E_{|W^s}$ preserves the fibers. We denote by $s^E : E_{|W^s} \rightarrow E_{|W^s}$ the linear operator associated to the action of $s$ on $E_{|W^s}$.

\begin{defi}Let $s\in L$. The $s$-equivariant Chern character of an $L$-invariant super-connection $\mathbb{A}$ is defined by: 
$$\Ch_s(\mathbb{A},X)=\Str \Big( s^E\cdot e^{-F_{\mathfrak{l}}(X)_{|W^s}} \Big) \in \mathcal{A}^\infty_{L(s)}(\mathfrak{l}(s),W^s),$$ where $F_{\mathfrak{l}}$ is the equivariant curvature of $\mathbb{A}$ and $X \in \mathfrak{l}$. 
\end{defi}
\noindent
Of course when the group $L$ is trivial then the $s$-equivariant Chern character is just the usual Chern character.

\begin{defi}
Let $\mathcal{V} \rightarrow W$ be a real $L$-equivariant vector bundle on $W$. Assume that $\mathcal{V}\rightarrow W$ is equipped with a $L$-invariant connection $\nabla $ with equivariant curvature $R(X)$, then the equivariant $\hat{A}$-genus $\hat{A}(\mathcal{V} )$ is defined by  $$\hat{A}(\mathcal{V} )(X)=\det^{1/2} \bigg(\dfrac{R(X)}{e^{R(X)/2}-e^{-R(X)/2}}\bigg), $$
which makes sense for $X$ in a small enough neighborhood of $0\in \mathfrak{l}$.
\end{defi}

\begin{defi}
Let $s\in L$. Denote by $N$ the normal bundle to $W^s$ in $W$ and $R_N(X)$, $X\in \mathfrak{l}(s)$, the equivariant curvature of $N$ with respect to a $L(s)$-invariant connection. The element 
$$D_s(N,X)=\det\big( 1-s^N\exp(R_N(X))\big),~ X\in \mathfrak{l}(s)$$
is a $L(s)$-equivariant closed form on $W^s$. 
\end{defi}

\begin{defi}
Let $E \rightarrow W$ be an Euclidean oriented vector bundle with orientation $o$. Denote by $F(X)$ the equivariant curvature associated to a $L$-invariant metric connection. The equivariant Euler class is defined by
$$\mathrm{Eul}_o(E)(X)=(-2\pi)^{rg(E)/2}\det^{1/2}_o(F(X)),$$
where $\det^{1/2}_o$ means the Pfaffian given by the orientation of $E$.
\end{defi}

\begin{remarque}
The forms $\Ch_s(\mathbb{A})$, $\hat{A}(\mathcal{V})$, $D_s(N)$ and $\mathrm{Eul}_o(E)$ are all equivariantly closed. The classes of $\hat{A}(\mathcal{V})$ and $D_s(N)$ do not depend on the connection. The class of $\mathrm{Eul}_o(E)$ only depends on the chosen orientation $o$ of $E$.
\end{remarque}

\subsection{Equivariant cohomology with generalised coefficients}\label{section:equi:coh:coefgen}
We recall some cohomological constructions \cite{kumar1993equivariant}.
Let $L$ be a compact Lie group with Lie algebra $\mathfrak{l}$. Let $W$ be a $L$-manifold. 
Let us recall the definition of the equivariant cohomology with generalised coefficients 
\cite{duflo1990orbites}, see also \cite{kumar1993equivariant}. Let $C^{-\infty }(\mathfrak{l},\mathcal{A}(W))$ be the space of generalised functions on $\mathfrak{l}$ with values in $\mathcal{A}(W)$. By definition, this is the space of continuous linear maps from the space $\mathcal{D}(\mathfrak{l})$ of $C^{\infty}$ densities with compact support on $\mathfrak{l}$ to $\mathcal{A}(W)$, where $\mathcal{D}(\mathfrak{l})$ and $\mathcal{A}(W)$ are equipped with the $C^{\infty}$ topologies. So if $\alpha \in C^{-\infty }(\mathfrak{l},\mathcal{A}(W))$ and if $\phi \in \mathcal{D}(\mathfrak{l})$ then $\langle \alpha ,\phi \rangle $ is a differential form on $W$ denoted by $\int_\mathfrak{l} \alpha(X)\phi(X)dX$. A $C^\infty $ density with compact support on $\mathfrak{l}$ is also called a test density, and a $C^{\infty}$ function with compact support on $\mathfrak{l}$ is called a test function. Denote by $E^i$ a basis of $\mathfrak{l}$ and $E_i$ its dual basis. Let $d$ be the operator on $C^{-\infty }(\mathfrak{l},\mathcal{A}(W))$ defined by $$\langle d\alpha , \phi \rangle =d\langle \alpha ,\phi \rangle, ~\mathrm{pour} ~\phi\in \mathcal{D}(\mathfrak{l}).$$
Let $\iota$ be the operator defined by
$$\langle \iota \alpha ,\phi \rangle = \sum\limits_i \iota((E^i)^*_W)\langle \alpha ,E_i \phi \rangle ,$$
where $(E^i)^*_W$ means as usual the vector field generated  by $E^i\in \mathfrak{l}$ on $W$ and where $E_i\phi$ means the tensor product $E_i \otimes \phi $. 
Let then $d_\mathfrak{l}$ be the operator on $C^{-\infty }(\mathfrak{l},\mathcal{A}(W))$ defined by 
$$d_\mathfrak{l}\alpha=d\alpha - \iota \alpha .$$
The operator $d_\mathfrak{l}$ coincides with the equivariant differential on $C^{\infty }(\mathfrak{l},\mathcal{A}(W))\subset C^{-\infty }(\mathfrak{l},\mathcal{A}(W))$. The group $L$ naturally acts on $C^{-\infty }(\mathfrak{l},\mathcal{A}(W))$ by $\langle g\cdot \alpha , \phi \rangle = g \cdot \langle \alpha ,g^{-1} \cdot \phi \rangle .$ The action of $L$ commutes with the operators $d$ and $\iota $. The space of generalized functions on $\mathfrak{l}$ with values in $\mathcal{A}(W)$ which are $L$-equivariant  is denoted by 
$$\mathcal{A}^{-\infty}_L(\mathfrak{l} , \mathcal{A}(W))=C^{-\infty }(\mathfrak{l},\mathcal{A}(W))^L.$$
The operator $d_\mathfrak{l}$ preserves $\mathcal{A}^{-\infty}_L(\mathfrak{l},W)$ and satisfies $d_\mathfrak{l}^2=0$.
Similarly, if we replace $\mathcal{A}(W)$ by $\mathcal{A}_c(W)$ the space of compactly supported forms then we can define $\mathcal{A}^{-\infty}_{c,L}(\mathfrak{l},W)=C^{-\infty}(\mathfrak{l},\mathcal{A}_c(W))^L$. \\
We also need to consider $L$-equivariant generalized forms which are defined on an open neighbourhood of the origin in $\mathfrak{l}$.
If $O$ is an $L$-invariant open subset of $\mathfrak{l}$, we denote by $\mathcal{A}^{-\infty }_{L}(O,W)$ and $\mathcal{A}^{-\infty }_{c,L}(O,W)$ the spaces obtained similarly.  \\
Let $U$ be a $L$-invariant open set in $W$.  The space of forms with generalized coefficients and with support in $U$ is denoted by $\mathcal{A}^{-\infty}_U(O,W)$. This is the space of differential forms with generalized coefficients such that there is a $L$-invariant closed subspace $C_\alpha \subset U$ such that $\int \alpha (X)\phi (X) dX$ is supported in $C_\alpha$ for any test density $\phi$. 

\begin{notation}
The cohomology of the complex $(\mathcal{A}^{-\infty}_L(\mathfrak{l},W) ,d_\mathfrak{l})$ is denoted by $\mathcal{H}^{-\infty}_L(\mathfrak{l},W)$.\\ 
The cohomology of the complex $(\mathcal{A}^{-\infty}_{c,L}(\mathfrak{l},W) ,d_\mathfrak{l})$ is denoted by $\mathcal{H}^{-\infty}_{c,L}(\mathfrak{l},W)$.\\
The cohomology of the complex $(\mathcal{A}^{-\infty}_{L}(O,W) ,d_\mathfrak{l})$ is denoted by $\mathcal{H}^{-\infty}_{L}(O,W)$.\\
The cohomology of the complex $(\mathcal{A}^{-\infty}_{c,L}(O,W) ,d_\mathfrak{l})$ is denoted by $\mathcal{H}^{-\infty}_{c,L}(O,W)$.\\
The cohomology of the complex $(\mathcal{A}^{-\infty}_{U}(O,W) ,d_\mathfrak{l})$ is denoted by $\mathcal{H}^{-\infty}_{U}(O,W)$.
\end{notation}

There is a natural map $$\mathcal{H}^{\infty}(\mathfrak{l},W)\rightarrow \mathcal{H}^{-\infty}(\mathfrak{l},W)$$ induced by the inclusion $\mathcal{A}^{\infty}_L(\mathfrak{l},\mathcal{A}(W)) \hookrightarrow \mathcal{A}^{-\infty}_L(\mathfrak{l},\mathcal{A}(W))$.
If $p : M \rightarrow B$ is a oriented $L$-equivariant fibration, then integration along the fibers  
$\int_{M|B}$ defines a map from $\mathcal{A}^{-\infty}_{c,L}(\mathfrak{l},M)$ to $\mathcal{A}^{-\infty}_{c,L}(\mathfrak{l},B)$:
$$\langle \int_{M|B}\alpha , \phi \rangle := \int_{M|B}\langle \alpha ,\phi \rangle, ~\forall \phi \in \mathcal{D}(\mathfrak{l}),$$ and induces a well defined map:
$$\int_{M|B} :  \mathcal{H}^{-\infty}_{c,L}(\mathfrak{l},M)\rightarrow \mathcal{H}^{-\infty}_{c,L}(\mathfrak{l},B).$$
Finally note that if $\alpha \in \mathcal{H}^{\infty}_{c,L}(\mathfrak{l},M)$, and $\beta \in \mathcal{H}^{-\infty}_{c,L}(\mathfrak{l},B)$ then $\alpha \wedge p^* \beta \in \mathcal{H}^{-\infty}_{c,L}(\mathfrak{l},M)$ and
$$\int_{M|B}\alpha \wedge p^\beta =(\int_{M|B}\alpha)\wedge \beta .$$
Let us recall some fact about restriction of generalized functions \cite{paradan2008index,kumar1993equivariant}. \\
Let $s\in G$ and let $U_s(0)$ be an open $L(s)$-invariant neighborhood of $0$ in $\mathfrak{l}(s)$ such that the map $[g,Y]\mapsto gse^Yg^{-1}$ is an open embedding of $L\times_{L(s)}U_s(0)$ on an open neighborhood of the conjugacy class $L\cdot s =\{gsg^{-1}, \ g\in L\} \simeq L/L(s)$.\\
Let $S\in \mathfrak{l}$ and let $U_S(0)$ be an open $L(S)$-invariant neighborhood of $0$ in $\mathfrak{l}(s)$ such that the map $[g,Y]\mapsto Ad(g)(S+Y)$ is an open embedding of $L\times_{L(s)}U_S(0)$ on an open neighborhood of the adjoint orbit $L\cdot S \simeq L/L(S)$.\\
Let $\Theta \in C^{\infty}(L)^{Ad(L)}$ be a $Ad(L)$-invariant generalized function on $L$. For any $s\in L$, $\Theta $ defines a $Ad(L)$-invariant generalized function on $L \times_{L(s)} U_s(0) \hookrightarrow L$ which admits a restriction to the submanifold $U_s(0)$ denoted by
 $$\Theta_{\|s} \in C^{-\infty}(U_s(0))^{L(s)}$$
in \cite{paradan2008index}. If $\Theta$ is smooth then $\Theta_{\|s}(Y)=\Theta(se^Y)$. \\
Similarly, let $\theta \in C^{\infty}(\mathfrak{l})^{Ad(L)}$ be a $Ad(L)$-invariant generalized function on $\mathfrak{l}$. For any $S\in \mathfrak{l}$, $\theta $ defines a $Ad(L)$-invariant generalized function on $L \times_{L(S)} U_S(0) \hookrightarrow \mathfrak{l}$ which admits a restriction to the submanifold $U_S(0)$ denoted by
 $$\theta_{\|S} \in C^{-\infty}(U_S(0))^{L(S)}$$
in \cite{paradan2008index}. If $\theta$ is smooth then $\theta_{\|S}(Y)=\theta(S+Y)$. 
We have $L(se^S)=L(s)\cap L(S), ~\forall S\in U_s(0)$. Let $\Theta\|_s \in C^{-\infty}(U_s(0))^{Ad(L(s))}$ be the restriction of a generalized function $\Theta \in C^{-\infty}(L)^{Ad(L)}$. For any $S\in U_s(0)$, the generalized function $\Theta\|_s$ admits a restriction $(\Theta\|_s)\|_S$ which is a $Ad(L(se^S))$-invariant generalized function defined in a neiborhood of $0 \in \mathfrak{l}(s)\cap \mathfrak{l}(S)=\mathfrak{l}(se^S)$.

\begin{lem}\cite{DV:comoEquiDescente}
Let $\Theta \in C^{-\infty}(L)^{Ad(L)}$.
\begin{itemize}
\item For $s\in L$ and $S\in U_s(0)$, we have the following equality of generalized functions defined in a neighborhood of $0 \in \mathfrak{l}(se^S)$
$$(\Theta\|_s)\|_S=\Theta\|_{se^S}.$$
When $\Theta \in C^{-\infty}(L)^{Ad(L)}$ is smooth this condition is easy to check: for $Y\in \mathfrak{l}(se^S)$, we have 
$$(\Theta\|_s)\|_S(Y)=\Theta\|_s(S+Y)=\Theta(se^{S+Y})=\Theta(se^Se^Y)=\Theta\|_{se^S}(Y).$$
\item Let $s, k\in L$. We have the following equality of generalized functions defined in a neighborhood of $0 \in \mathfrak{l}(s)$ 
$$(\Theta\|_s)\|_S=\Theta\|_{ksk^{-1}}\circ Ad(k).$$
\end{itemize}
\end{lem}

\begin{thm}\cite{DV:comoEquiDescente}
Let $\theta_s \in C^{-\infty}(U_s(0))^{Ad(L(s))}$ be a family of generalized functions. Assume that the following conditions are verified.
\begin{itemize}
\item Invariance : $\forall k, s\in L$, we have the following equality of generalized functions defined in a neighborhood of $0\in \mathfrak{l}(s)$ 
$$\theta_s=\theta_{ksk^{-1}}\circ Ad(k).$$
\item Compatibility : $\forall s \in L$ and $S\in U_s(0)$, we have the following equality of generalized functions defined in a neighborhood of $0\in \mathfrak{l}(se^S)$ 
$$\theta_s\|_S=\theta_{se^S}.$$
\end{itemize}
Then there exists a unique generalized function $\Theta \in C^{-\infty}(L)^{Ad(L)}$ such that, for any $s\in L$, the equality $\Theta\|s=\theta_s$ holds in $C^{-\infty}(U_s(0))^{Ad(L(s))}$.
\end{thm}
For details on restrictions of invariant generalized functions see for instance \cite{DV:comoEquiDescente, paradan2008index}.
\subsection{Chern character of a morphism\cite{paradan2008equivariant}}\label{section:Chern:Character}

In this section, we recall the Chern character of a morphism \cite{paradan2008equivariant}. Recall that the Atiyah-Singer index formula involves integration over the non compact manifold $T^*M$ and so to perform integration the representative of the Chern character has to be compactly supported on $T^*M$. In the equivariant case, Berline and Vergne defined a Chern character for what they call \emph{good symbols} in the context of elliptic and $G$-transversally elliptic symbols \cite{BV:ChernCharacterTransversally}. In this context, they get a Chern character with some decreasing properties which allows them to permorf integration. The goal of the construction of \cite{paradan2008equivariant} is to define a Chern character compactly supported for non elliptic symbol which are transversally elliptic with respect to a group action without assuming that the symbol is a \emph{good symbol}. Paradan and Vergne get a compactly supported Chern character for any $G$-transversally elliptic symbol in \cite{paradan2008equivariant}, but to obtain that they have to use equivariant cohomology with generalized coefficients on the Lie algebra of $G$. Let $L$ be a compact Lie group and denote again by $\mathfrak{l}$ its Lie algebra. Recall that the analytical index of Atiyah is a distribution on the group $L$ which is $Ad(L)$-invariant so it is natural to define a Chern character with  coefficients generalized functions on $\mathfrak{l}=Lie(L)$.\\ 
Let $W$ be a $L$-manifold. Let $\lambda$ be a real $L$-invariant $1$-form on $W$. For any $w\in W$, we have $\lambda(w)\in T^*_wW$.

\begin{defi}

The $1$-form $\lambda$ defines an equivariant map 
$$f_\lambda : W \rightarrow \mathfrak{l}^* ~\mathrm{given\ by} <f_\lambda(w),X>=<\lambda(w),X^*_W(w)>. $$
\end{defi}

\noindent 
Denote by $C_\lambda$ the $L$-invariant closed subspace of $W$ given by:
$$C_\lambda=\{f_\lambda = 0\}.$$
Note that when $\lambda$ is the Liouville $1$-form on $T^*M$ then $f_\lambda=T^*_LM$. 
Let $\sigma : E^+ \rightarrow E^-$ be a $L$-equivariant morphism on $W$, and denote by $C_{\lambda,\sigma}$ the invariant closed subspace given by
$$C_{\lambda,\sigma}=C_\lambda\cap \mathrm{supp}(\sigma) ~\mathrm{and}~v_\sigma=\begin{pmatrix}
0&\sigma^*\\
\sigma &0
\end{pmatrix}. $$  
\noindent
In the construction of \cite{paradan2008equivariant}, the super-connection $\mathbb{A}^{\sigma,\lambda}$ used to define the Chern character is a combinaison of a super-connection $\mathbb{A}$, the symbol $\sigma$ and the $1$-form $\lambda$. Here $\sigma$ reduces the support of the Chern character to the support of $\sigma $ and $\lambda$ reduces the support of the Chern character to $C_\lambda$ and so the Chern character is supported in $C_{\lambda,\sigma}$.  
Let $\mathbb{A}$ be a $L$-invariant super-connection on $E$. 
We will use the following notations (see \cite{paradan2008equivariant}):  
\begin{enumerate}
\item $\mathbb{A}^{\sigma,\lambda}(t)=\mathbb{A}+it(v_\sigma+\lambda)$, $t\in \mathbb{R}$;
\item $F(\sigma,\lambda,\mathbb{A},t)(X)=-t^2v_\sigma^2-it<f_\lambda,X>+\mu^\mathbb{A}(X)+it[\mathbb{A},v_\sigma]+\mathbb{A}^2+itd\lambda$, which is the equivariant curvature of $\mathbb{A}^{\sigma,\lambda}(t)$;
\item $\eta_s(\sigma,\lambda,\mathbb{A},t)(X)=-\Str \bigg(i(v_\sigma+\lambda )s^Ee^{F(\sigma,\lambda,\mathbb{A},t)(X)}_{|W^s}\bigg)=-e^{itd_\g\lambda (X)}\Str\bigg(i(v_\sigma +\lambda )s^Ee^{F(\sigma,\mathbb{A},t)(X)}_{|W^s}\bigg)$, which is the transgression form associated to the Chern character of $\mathbb{A}^{\sigma ,\lambda}(t)$ ;
\item $\beta_s (\sigma,\lambda ,\mathbb{A})=\int^\infty_0 \eta_s (\sigma,\lambda,\mathbb{A},t)dt$, which is a form with generalized coefficients since the convergence of $\int^T_0 \eta_s (\sigma,\lambda,\mathbb{A},t)dt$ when $T$ goes to infinity makes sense as a distribution.
\end{enumerate}

For the following theorem see \cite{paradan2008equivariant} and also Section 3.3 of \cite{paradan2008index}.

\begin{thm}$\mathrm{\cite[Theorem\  3.19]{paradan2008equivariant}}$\label{thm:chern:paradan}~\\
$\bullet$ For any $L(s)$-invariant open neighborhood $U$ of $C_{s,\lambda,\sigma}=C_{\lambda,\sigma}\cap W^s$, let $\chi \in C^{\infty}(W^s)^{L(s)}$ be a $L(s)$-invariant function which is equal to $1$ in a neighborhood of $C_{s,\lambda,\sigma}$ and with support contained in $U$. 
\begin{enumerate}
\item Then $c_s(\sigma,\lambda ,\mathbb{A},\chi)=\chi\Ch_s(\mathbb{A})+d\chi \beta_s(\sigma,\lambda,\mathbb{A})$
is an equivariant closed differential form with generalized coefficients supported in $U$. Furthermore, we have
$$c_{se^X}(\sigma ,\lambda ,\mathbb{A},\chi)(Y)=c_s(\sigma ,\lambda, \mathbb{A}, \chi)(X+Y)_{|W^s\cap W^X},$$
for any $X\in \mathfrak{l}(s)$ and $Y \in \mathfrak{l}(s)\cap \mathfrak{l}(X)$.
\item  The cohomology class $c_U(\sigma,\lambda ,s )\in \mathcal{H}^{-\infty}_U(\mathfrak{l}(s),W^s)$ of $c_s(\sigma,\lambda ,\mathbb{A},\chi)$ does not depend on the choices of the super-connection $\mathbb{A},$ $\chi$ and the hermitian structures on $E^\pm$. 
\item Moreover, the inverse family $c_U(\sigma,\lambda)$ when $U$ runs over the neighborhood of $C_{s,\lambda,\sigma }$ defines a class
$$\Ch_\mathrm{sup}(\sigma,\lambda ,s )\in \mathcal{H}^{-\infty}_{C_{s,\lambda,\sigma}}(\mathfrak{l}(s),W^s),$$
where $\mathcal{H}^{-\infty}_{C_{s,\lambda,\sigma}}(\mathfrak{l}(s),W^s)$ is the projective limit of the following projective system $(\mathcal{H}^{-\infty}_U(\mathfrak{l}(s),W^s))_{C_{s,\lambda ,\sigma}\subset U}$.
\end{enumerate}
$\bullet$ The image $\Ch_{\sup }(\sigma,\lambda ,s )$ in $ \mathcal{H}^{-\infty}_{\mathrm{supp}(\sigma ) \cap W^s}(\mathfrak{l}(s),W^s)$ is equal to $\Ch_{\sup}(\sigma, s)$.\\
$\bullet $ Let $F$ be a $L$-invariant subspace of $W^s$. For $\tau\in [0,1]$, Let $\sigma_\tau :E^+ \rightarrow E^-$ be a differential family of $L$-equivariant smooth morphisms and let $\lambda_\tau$ be a $L$-invariant differential family of $1$-forms such that $C_{s,\lambda_\tau,\sigma_\tau} \subset F$ $\forall \tau\in [0,1]$. Then all the classes $\Ch_{\sup } (\sigma_\tau ,\lambda_\tau ,s )$ coincide in $\mathcal{H}_F^{-\infty}(\mathfrak{l}(s),W^s)$.
\end{thm}

\begin{defi}
If $C_{s,\lambda,\sigma}$ is a compact subspace of $W^s$, then we can define:
$$\Ch_c(\sigma,\lambda ,s )\in \mathcal{H}_c^{-\infty}(\mathfrak{l}(s),W^s)$$
as the image of $\Ch_{\sup}(\sigma , \lambda ,s )\in \mathcal{H}^{-\infty}_{C_{s,\lambda,\sigma}}(\mathfrak{l}(s),W^s)$ in $\mathcal{H}_c^{-\infty}(\mathfrak{l}(s),W^s)$. A representative of $\Ch_{c}(\sigma,\lambda ,s )$ is then given by any equivariant form $c_s(\sigma,\lambda,\mathbb{A},\psi)$ as before, with $\psi$ compactly supported. See again \cite{paradan2008equivariant}.
\end{defi}

\begin{remarque}
If $\sigma$ is a $L$-equivariant elliptic symbol then we get $\Ch_{c}(\sigma ,s )\in \mathcal{H}^{\infty}_c(\mathfrak{l}(s),W^s)$.

\end{remarque}

When the action of the group on the manifold is trivial then the $s$-equivariant Chern character can be compute using the usual Chern character with values in de Rham cohomology and the character morphism which associates to a representation its character. In fact, we have the following lemma from \cite{BV:formuleloc:Kirillov}.  

\begin{lem}\cite{BV:formuleloc:Kirillov}\label{lem : Ch_g=Chchi}
Assume that $L(s)$ acts trivially on a manifold $W$. We have the following isomorphism:
\begin{align*}
\k(W)\otimes R(L(s)) \rightarrow & \k_{\mathrm{L}(s)}(W)\\
[\sigma]\otimes V  \mapsto &[\sigma \otimes id_V],
\end{align*}
\begin{align*}
H(W,\mathbb{C})\otimes C^{\infty }(\mathfrak{l}(s))^{L(s)} \rightarrow & \mathcal{H}^{\infty }_{L(s)}(\mathfrak{l}(s) ,Y)\\
[\omega ]\otimes \varphi \mapsto &[X\mapsto \omega \varphi (X)]
\end{align*}
We denote by $\chi$ the character morphism which associates to a representation its character. The following diagram is commutative:
$$
\xymatrix{ 
\k_{\mathrm{L}(s)}(W) \ar[r]^{\Ch_s} \ar[d] & \mathcal{H}^{\infty }_{L(s)}(\mathfrak{l}(s), W) \\
\k(W) \otimes R(L(s)) \ar[r]_{\Ch \otimes \chi(se^{\bullet})} &H(W,\mathbb{C}) \otimes C^{\infty }(\mathfrak{l}(s))^{L(s)}. \ar[u]
}
$$

\end{lem}

\begin{proof} We use the previous notations. Let $\sigma : E^+ \rightarrow E^-$ be an elliptic morphism, let $\nabla^E$ be a graded connection on $E = E^+ \oplus E^-$ and $\nabla^V = d \otimes id_V$ be the trivial connexion on $Y \times V$. Denote by $\nabla$ the product connection on $E \otimes V$. Then the associated equivariant curvature to $\sigma \otimes  id_V$ is $F(\sigma \otimes id_V ,\nabla, t)(X)=F(\sigma \otimes id_V ,\nabla ,t) + \mu^{\nabla}(X)$ on $Y$ is equal to $F(\sigma ,\nabla^E ,t) \otimes 1 + 1\otimes X$ because $\mu^{\nabla}(X)=1 \otimes X$ since $X^*_W=0$. So we get $\Ch_s(\nabla)(X)= \Ch(\nabla^E)\chi(se^X)$ and $\eta_s(\sigma \otimes id_V,\nabla ,t)(X)=\eta(\sigma , \nabla^E ,t)\chi(se^X)$. And then if $\psi$ is a smooth invariant function on $Y$, we obtain
$$c_s(\sigma \otimes id_V, \nabla , \psi)=\psi \Ch(\nabla^E)\otimes \chi(se^\bullet) + d\psi \beta (\sigma ,\nabla^E)\otimes \chi(se^\bullet) = c(\sigma , \nabla^E,\psi )\otimes \chi(se^\bullet).$$

\end{proof}

\subsection{The index class of a family of $G$-transversally elliptic operators}

Here we recall the definition of the index class of a family of $G$-transversally elliptic operators \cite{baldare:KK}. Denote by $T_GM=\{\alpha \in T^*M | \alpha (X^*_M)=0,\forall X\in\g\}$. Let $T^V_GM$ denote the space $T^VM\cap T_GM$. Recall that a family $P=(P_b)_{b\in B}$ of $G$-transversally elliptic pseudodifferential operators is a family of pseudodifferential operators such that its principal symbol $\sigma(P)$ is invertible on $T^V_GM\setminus M$. 
We fix from now on a $G$-invariant continuous family of Borel measures $(\mu_b)$ which are in the Lebesgue class, constructed using a partition of unity of $B$. So, for any $f\in C(M)$, the map $b\mapsto \int_{M_b}f(m)d\mu_b(m)$ is continuous, and each measure $\mu_b$ is fully supported in the fiber $M_b$.
Since  $E$ is equipped with a hermitian structure, the $C(B)$-modules $C^{\infty ,0 }(M,E^\pm)$ of continuous fiberwise smooth sections over $M$, are naturally equipped with the structure of  pre-Hilbert right $G$-equivariant $C(B)$-modules with the inner product given by:
$$
\langle s,s'\rangle (b) =\int_{M_b}\langle s(m),s'(m)\rangle_{E^{\pm}_m}d\mu_b(m), \quad \text{ for } s, s'\in C^{\infty ,0 }(M,E^\pm).
$$
We denote by $\mathcal{E}^{\pm}$ the completion of $C^{\infty , 0}(M,E^\pm)$. So, $ \mathcal{E} = \mathcal{E}^+\oplus  \mathcal{E}^-$ is our $G$-equivariant $\mathbb{Z}_2$-graded Hilbert module on $C(B)$.
We denote by $C^*G$ the $C^*$-algebra associated with $G$.
Let $\pi : C^*G \rightarrow \mathcal{L}_{C(B)}(\mathcal{E})$ be the $\ast$-representation given by 
$$\pi(\varphi)s=\int_G\varphi(g)(g\cdot s)dg,~\forall \varphi \in L^1(G)~\mathrm{and}~ s\in C^{\infty ,0}(M,E).$$
If $P_0 : C^{\infty ,0}(M,E^+ )\rightarrow C^{\infty ,0}(M,E^-)$ is a family of pseudodifferential operators of order $0$, we denote by $P$ the family $\begin{pmatrix}
0&P_0^*\\
P_0&0
\end{pmatrix}$.\\
Let $H$ be a compact Lie group. Assume that $H$ acts on $M$ and that $E$ and $P$ are also $H$-equivariant. 
\begin{defi}\cite{baldare:KK}
The index class $\ind (P_0)$ of a $G\times H$-invariant family $P_0$ of $G$-transversally elliptic operators is defined by:
$$\mathrm{Ind^{M|B}_H} (P_0)=[\mathcal{E},\pi,P] \in \k\Khh(C^*G,C(B)).$$
If $H$ is the trivial group then we simply denote by $\ind(P_0)$ the index class which leaves in $\k\k(C^*G,C(B)).$  
\end{defi}

Let us recall the definition of $\k$-multiplicity of an irreductible unitary representation of $G$. Denote by $\hat{G}$ the space of isomorphism classes of unitary irreducible representations of $G$. 
\begin{defi}\cite{baldare:KK}
The $\k$-multiplicity $m_P(V)$ of a irreducible unitary representation $V$ of $G$ in the index class  $\ind (P_0)$ is  the image of the class 
$[(\mathcal{E}_V^G,P_V^G)]\in \k\k(\mathbb{C},C(B))$ under the isomorphism $\k\k(\mathbb{C},C(B))\cong \k(B)$. So $m_P(V)$ is the class of a virtual vector bundle over $B$, an element of the topological $\k$-theory group $\k(B)$. The class $[(\mathcal{E}_V^G,P_V^G)]$ coincides (as expected) with the Kasparov product
$$
[V]\underset{C^*G}{ \otimes}\ind(P_0)\in \k\k(\mathbb{C},C(B)),
$$
\end{defi}

As we have $\k\k(C^*G,C(B))\cong\mathrm{Hom}(R(G),\k(B))$ (see for instance \cite{rosenberg1987kunneth}), we get:

\begin{prop}
The index class of a $G$-invariant family $P_0$ of $G$-transversally elliptic operators is totally determined by its multiplicities and we have:
$$\mathrm{Ind^{M|B}}(P_0)=\sum\limits_{V\in \hat{G}}m_P(V)\chi_V.$$ 

\end{prop}

\begin{proof}
We simply apply the Universal Coefficient Theorem in bivariant $\k$-theory \cite{rosenberg1987kunneth}. Indeed, the $C^*$-algebra of $G$ belongs to the bootstrap category, see for example \cite{meyer2002comparisons}. Moreover, $R(G)$ is a free module which is isomorphic to $\bigoplus \limits_{V\in \hat{G}} \mathbb{Z}$.
\end{proof}

\subsection{Index theorem for $H$-equivariant elliptic families \cite{Atiyah-Singer:IV}}

\noindent Let $H$ be a compact group. Let $P_0$ be a $H$-equivariant continuous family of pseudodifferential operators of order $0$ which is elliptic along the fibers. Recall that a family $P$ is elliptic along the fibers if its principal symbol $\sigma_P$ given on each fiber by $\sigma_{P_b}$ is invertible on $T^VM^*\setminus M$.
\begin{defi}\cite{Atiyah-Singer:IV}\label{def:indice:analytique}   
The $H$-equivariant analytical index $\mathrm{Ind^{M|B}_H}(P_0)$ of $P_0$ is defined as the image in $\Khh(B)$ of the index class $[\mathcal{E},P] \in \k\Khh(\mathbb{C},C(B))$. It is the formal difference of the continuous fields of Hilbert spaces $\big(\ker((P_0)_b)\big)_{b\in B}$ and $\big(\ker((P_0^*)_b)\big)_{b\in B}$.  
\end{defi}

Here $\mathcal{E}$ and $P$ are constructed as in the previous section.
Denote by $\pi$ the projection $T^VM \rightarrow M$.
The map $p\circ \pi$ is $\Khh$-oriented
so it defines an element $p_! \in \k\Khh(C_0(T^VM),C(B))$ \cite{Connes:Skandalis:longIndThmFoliations}.

\begin{defi}\cite{Atiyah-Singer:IV}
Let $P_0$ be a $H$-equivariant elliptic family of operators on $M$ parametrized by a compact manifold $B$.
The topological index is defined by $\mathrm{Ind^{M|B}_{H,t}}(P_0)=[\sigma (P_0)] \otimes_{C_0(T^VM)} p_! \in \Khh(B)$.
\end{defi}

\noindent Let us recall the index theorem for families \cite{Atiyah-Singer:IV}.

\begin{thm}\cite{Atiyah-Singer:IV}
The analytical index $\mathrm{Ind^{M|B}_H}$ and the topological index $\mathrm{Ind^{M|B}_{H,t}}$ coincide.
\end{thm}

\begin{remarque}
Our notation for the index class of a $G\times H$-family of $G$-transversally elliptic operators is consistent with the equivariant case because if $G$ is the trivial group then the family is $H$-equivariant and elliptic and then its index $\mathrm{Ind^{M|B}_H}(P_0)$ lives in $\k\Khh(\mathbb{C},C(B))$.
\end{remarque}

\section{The Chern character of the index class }\label{Chern:caracter:index:class}\label{chapitre 2}
We want to obtain cohomological formulas for the index of a $G$-invariant family of $G$-transversally elliptic operators.
We will use the bivariant local cyclic homology \cite{puschnigg2003diffeotopy}.

\subsection{The Chern character in bivariant local cyclic homology}
\noindent
Recall the following theorem from \cite{meyer2002comparisons}.

\begin{thm}[Universal Coefficient Theorem]\label{thm:UTC:HL}\cite{meyer2002comparisons}
Let $A$ and $B$ be separable $C^*$-algebras $\k\k$-equivalent to commutative separable $C^*$-algebras. Then there is a natural isomorphism 
$$
\mathrm{HL}(A,B)\simeq \mathrm{Hom}(\k_*(A)\otimes_\mathbb{Z} \mathbb{C},\k_*(B)\otimes_\mathbb{Z} \mathbb{C})
$$
of graded vector spaces.
\end{thm}
As the $C^*$-algebra $C^*G$ of a compact Lie group $G$ is in the bootstrap category \cite{meyer2002comparisons}, we can apply the universal coefficients theorem in bivariant local cyclic homology. 
Recall that 
$$\mathrm{HL}_*(C^*G) \simeq \left\{\begin{array}{lll} R(G)\otimes_\mathbb{Z} \mathbb{C},& *=0\\
0,& *=1
\end{array}\right.\ \cite{meyer2002comparisons}$$
and $$\mathrm{HL}_*(C(B))\simeq \bigoplus \limits_n H^{*+2n}(B,\mathbb{C}) \ \cite{puschnigg2003diffeotopy}.$$

So we obtain:
\begin{prop}\label{cor:UTC:HL:C^*G:C(B)}
The bivariant local cyclic homology $\mathrm{HL}(C^*G,C(B))$ of the couple $(C^*G,C(B))$ is isomorphic to $\mathrm{Hom}(\mathrm{HL}(C^*G),\mathrm{HL}(C(B))$.
\end{prop}
\begin{prop}
The Chern character in bivariant local cyclic homology of the index class $\ind(P_0)$ of a $G$-invariant family $P_0$ of $G$-transversally elliptic operators is given by:
$$\Ch (\ind(P_0))=\sum \limits_{V \in \hat{G}} \Ch(m_P(V))\chi_V,$$
where $\Ch(m_p(V)) \in H(B,\mathbb{C})$ is the usual Chern character of $m_P(V)\in \k(B)$.
\end{prop}
\begin{proof}
By Proposition \ref{cor:UTC:HL:C^*G:C(B)}, we have $\mathrm{HL}(C^*G,C(B))\simeq\mathrm{Hom}(\mathrm{HL}(C^*G),\mathrm{HL}(C(B))$. So we deduce that $\Ch(\ind(P_0))$ is totally determined by its values on the irreductible representations. We know that the Chern character in bivariant local cyclic homology is compatible with the Kasparov's intersection product \cite{puschnigg2003diffeotopy}. Thus we have
$$\Ch(\ind(P_0))\circ\Ch(V)=\Ch([V]\otimes_{C*G}\ind(P_0))=\Ch(m_P(V)).$$
\end{proof}

\subsection{Distributional index with value in $H(B,\mathbb{C})$}
Following \cite{atiyah1974elliptic}, we want to obtain a distributional Chern character. More precisely, if $\varphi \in C^{\infty}(G) $ we want to know if the series $\sum \limits_{V\in \hat{G}}\Ch(m_P(V))<\chi_V,\varphi>_{L^2(G)}$ is convergent. Here the series has values in $ H(B,\mathbb{C}) $ while in \cite{atiyah1974elliptic} the series has values in $ \mathbb{C}$.
We show that the Chern character of the index class converges as a distribution on $C^{\infty}(G)$ with value in the finite dimensional cohomology of $B$. 
\begin{thm}\cite{baldare:KK}\label{thm:couplage:ind}
The pairing with values in $\k\k(C^*G,\mathbb{C})$ of the index class of a $G$-invariant family of pseudodifferential operator $P_0$ of order $0$ which is $G$-transversally elliptic along the fibers, with an element of the $\k$-homology group $\k\k(C(B),\mathbb{C})$ is given by the index class of a $G$-invariant pseudodifferential operator which is $G$-transversally elliptic on the ambiant manifold $M$.  
\end{thm}

\begin{cor}\label{cor:couplage:ind:Atiyah}
If $\alpha $ is an element of the $\k$-homology of $B$ then
\begin{enumerate}
\item the class $[\mathcal{E},\pi, P]\otimes_{C(B)} \alpha \in \k\k(C^*G,\mathbb{C})\simeq \mathrm{Hom}(R(G),\mathbb{C})$ is given by the distributionnal index of Atiyah \cite{atiyah1974elliptic}, i.e. the multiplicities of $ \ind(P_0)\otimes_{C(B)} \alpha $ are summable in the sens of distributions on $G$;
\item denote by $m([V]\otimes_{C*G}\ind(P_0)\otimes_{C(B)}\alpha)$ the integer associated to the multiplicity of $V$ $[V]\otimes_{C*G}\ind(P_0)\otimes_{C(B)}\alpha$ in $\ind(P_0)\otimes_{C(B)} \alpha $. For any $\varphi \in C^\infty(G)$, the series
$$\sum \limits_{V\in \hat{G}}m \big([V]\otimes_{C*G}\ind(P_0)\otimes_{C(B)}\alpha \big)<\chi_V,\varphi>_{L^2(G)},$$
is convergente in $\mathbb{C}$. 
\end{enumerate}  
\end{cor}

\begin{proof}
Indeed, by the Theorem \ref{thm:couplage:ind} $\ind(P_0)\otimes_{C(B)}\alpha $ is represented by the index class of a $G$-transversally elliptic pseudodifferential operator on $M$. By \cite{atiyah1974elliptic}, we know that the distributional index of a $G$-transversally elliptic operator $Q$ is tempered on $G$ and is totally determined by its multiplicities, that is to say $\mathrm{Ind}^M(Q)=\displaystyle \sum\limits_{V\in \hat{G}}m_Q(V)\chi_V$ is well defined as a distribution on $G$.
\end{proof}

We now deduce from the previous discussion that the Chern character of the index class of a $G$-invariant family of $G$-transversally elliptic pseudodifferential operators is a distribution on $G$ with value in the cohomology of the base $B$. For this purpose, the bivariant multiplicative Chern character is used \cite{puschnigg2003diffeotopy}. Then we show that the formal sum $\displaystyle \sum \limits_{V\in \hat{G}}\Ch(m_P(V))\chi_V$ converges in the distributional sense with value in the finite dimensional vector space $H(B,\mathbb{C})$.

\begin{thm}\label{thm:convergence:chern:indice}Assume that $B$ is oriented. Then the Chern character, in bivariant local cyclic cohomology, of the index class of a $G$-invariant family of pseudodifferential operators which is $G$-transversally elliptic along the fibers is a distribution with value in the even de Rham cohomology of $B$. We have more precisely,
$$\Ch(\ind(P_0))=\displaystyle \sum \limits_{V\in \hat{G}}\Ch(m_P(V))\chi_V \in C^{-\infty}(G,H^{2*}(B,\mathbb{C}))^{Ad(G)}.$$ 

\end{thm}

\begin{proof}
In order to show that $\Ch(\ind(P_0))=\sum \limits_{V\in \hat{G}}Ch(m_P(V))\chi_V$ converges in the sense of distributions in $C^{-\infty}(G,H^{2*}(B,\mathbb{C}))$, we will show that the pairing of any closed de Rham current $C$ on $B$,  with the Chern character of the index class, $\langle\Ch(\ind(P_0)),C\rangle$ is a distribution on $G$. Let $C\in H_*(B,\mathbb{C})$ be a de Rham current on $B$. As the Chern character is an isomorphism after tensoring by $\mathbb{C}$, we get that there exist $\lambda_1, \dots ,\lambda_n\in \mathbb{C}$ and $\alpha_1,\dots ,\alpha_n$ elements of $\k_*(B)$ such that $C=\sum \limits_{i=1}^{n}\lambda_i\Ch(\alpha_i)$. Then we have $\langle\Ch(\ind(P_0)),C\rangle=\sum \limits_{i=0}^n\lambda_i\langle\Ch(\ind(P_0)),\Ch(\alpha_i)\rangle$. Therefore it is sufficient to check that this pairing is a distribution for each element of the type $\Ch(\alpha)$, with $\alpha\in \k_*(B)$. Now if $\alpha\in \k_*(B)$ then we have $\langle\Ch(\ind(P_0)),\Ch(\alpha)\rangle=\Ch([\ind(P_0)]\otimes_{C(B)}\alpha)$ by multiplicativity of the Chern character in bivariant local cyclic homology. On the one hand, if $\alpha \in \k_1(B)$ then $[\ind(P_0)]\otimes_{C(B)}\alpha \in \k\k^1(C^*G,\mathbb{C})=0$ and so $<\Ch(\ind(P_0)),\Ch(\alpha)>=\Ch([\ind(P_0)]\otimes_{C(B)}\alpha)=0$. On the other hand, if $\alpha \in \k_0(B)$ then we know that by Corollary \ref{cor:couplage:ind:Atiyah} that $[\ind(P_0)]\otimes_{C(B)}\alpha=\Ch([\ind(P_0)]\otimes_{C(B)}\alpha)$ is a distribution on $G$. The Chern character is hence a distribution with values is the even de Rham cohomology because using the universal coefficient theorem in bivariant local cyclic homology, we have $\Ch(\mathcal{E},\pi,P)\in \mathrm{HL}(C^*G,C(B))\simeq \mathrm{Hom}(R(G)\otimes \mathbb{C},\k^0(B)\otimes \mathbb{C})\simeq \mathrm{Hom}(R(G)\otimes \mathbb{C},H^{2*}(B,\mathbb{C}))$.
\end{proof}

\begin{remarque}
If $B=\{\star\}$, we get the standard result (see \cite{natsume:nest}) that the Chern-Connes character of the index class of a $G$-invariant pseudodifferential operator which is $G$-transversally elliptic coincides with the distribution of Atiyah view as a trace on $C^\infty(G)^{Ad(G)}\simeq \mathrm{HP}_0(C^\infty(G))$, where $C^\infty(G)$ is viewed as convolution algebra.
\end{remarque}

\section{Berline-Vergne formula for equivariant families} \label{II}
In this section, we show a formula of the delocalized index in equivariant cohomology for $H$-invariant elliptic families. In a first paragraph, we begin by recalling the Bismut localization formula \cite{Bismut:localisation} which is a generalization of the Berline-Vergne localization formula \cite{BV:formuleloc:Kirillov,BGV}. 

\subsection{Review of the Bismut localization formula}
Assume that $M$ and $B$ are connected and oriented. Let $H$ be a compact Lie group and $\mathfrak{h}$ its Lie algebra. 
Let $X \in \mathfrak{h}$. Denote respectively by $X_M$ and $X_B$ the vector fields generated respectively by $X$ on $M$ and $B$. Let $B^X=\{b\in B/X_B(b)=0\}$ and $M^X=\{m\in M/X_M(m)=0\}$ be the submanifolds of zeros of $X_M$ and $X_B$. Since $p$ is $H$-equivariant we have $p_*X_M=X_B$. We know that the fibration $p : M \rightarrow B $ restricts in a fibration $p^X : M^X\rightarrow B^X$, with various, possibly empty, fibers over the different connected components of $B^X$, see for example \cite{Bismut:localisation} or \cite{heitsch1991} for foliation.

Choose a $H$-invariant metric $< \cdot ,\cdot >$ on $M$, i.e. 
$$\forall Y_1,~ Y_2 \in C^{\infty}(M,TM)~ <h\cdot Y_1, h\cdot Y_2>=<Y_1 , Y_2>,\quad \forall h\in H.$$
Let $T^{Hor}M$ be the orthogonal bundle of $T^VM$. Denote by $N^X_M$ the normal bundle of $M^X$ in $M$ and by $N^X_B$ the normal bundle of $B^X$ in $B$. We identify $N^X_M$ with the orthogonal of $TM^X$ in $TM$ using the riemanian metric. We can lift $N^X_B$ by $p_*$ into a subbundle of $T^{Hor}M$ on $M$. We note it $p^*N^X_B$.

\begin{thm}\cite{Bismut:localisation}\label{thm:bismut:fibréNormal}With the previous notation, we have:
\begin{enumerate}
\item The vertical normal bundle $N^X_M\cap T^VM$ of $M^X$ coincides with the normal bundle $\mathcal{N} (p^{-1}(b)\cap M^X, p^{-1}(b))$ of  $p^{-1}(b)\cap M^X$ in $ p^{-1}(b)$. 
\item The normal vertical bundle $N^X_M\cap T^VM$ of  $M^X$ coincide with the normal bundle $\mathcal{N}(M^X,p^{-1}(B^X))$ of  $M^X$ in $p^{-1 }(B^X)$.
\item Moreover, we have $N^X_M =N^X_M\cap T^VM \oplus N^X_B$.

\end{enumerate}
\end{thm}

\begin{lem}\label{lem:orientation:fibreN}\cite{kobayashi1958}
Let $T$ be a torus. Let $V$ be a $T$-manifold. Denote by $V^T$ the set of fixed points of $V$ under the action of $T$. Then the normal bundle $N$ of $V^T$ in $V$ is oriented by a complex structure. In this case, we can define the equivariant Euler class $\Eul(N)$ of the normal bundle $N$.

\end{lem}

Recall that the de Rham differential and the contraction by vector field satisfy the following $d^B \circ \int_{M|B} = \int_{M|B}\circ d^M$ and $\iota(X^*_B) \circ \int_{M|B} = \int_{M|B}\circ \iota(X^*_M)$, so $d_X^B \circ \int_{M|B} = \int_{M|B}\circ d_X^M$, see for example \cite[Chapter 10]{Guillemin:Sternberg:SupersymEquiv} or \cite{duflo1990orbites,Bismut:localisation}.

\begin{thm}[Bismut localization formula, 1986 \cite{Bismut:localisation}]\label{thm : Bismut}~\\
Denote by $j : B^X \hookrightarrow B $, $i : M^X \hookrightarrow M$, $\Eul(N^X_M\cap T^VM,X)$ the equivariant Euler form of $N^X_M\cap T^VM$ and $\int_{M|B}$ the integration along the fibers \cite{Bott:Tu}. 
Let $\alpha \in \mathcal{A}(M)$ be a differential form which is $(d+\iota (X))$-closed. Then the following equality holds in the de Rham cohomology $H(B^X,\mathbb{C})$ since $d_X=d$ on $B^X$:
$$j^* \int_{M|B}\alpha =\int_{M^X|B^X} \dfrac{i^*\alpha}{\Eul(N^X_M\cap T^VM,X)}.$$

\end{thm}

\subsection{Delocalized index formulas}
Let again $H$ be a compact Lie group and $\mathfrak{h}$ its Lie algebra. Let $p : M \rightarrow B$ be a $H$-equivariant fibration of compact manifolds. Assume that $B$ is oriented. Assume that $H$ is a topologically cyclic group generated by a single element $h$. We can always reduce to this case by using Remark \ref{rem:remplacerGroupe} below. Denote by $i : T^VM^H \hookrightarrow T^VM$ and $j : B^H \hookrightarrow B$ the inclusions of the fixed point submanifolds. Let us recall that when $W$ is a manifold with trivial action then $\k_\mathrm{H}(W) \simeq \k(W)\otimes R(H)$ and so we have an evaluation morphism $\k_\mathrm{H}(W) \rightarrow \k(W)\otimes \mathbb{C}$ using the previous isomorphism and the character morphism $\chi$. Then we can apply the Chern character tensored by $\mathbb{C}$ to such evaluation. For $u\in \k_\mathrm{H}(W)$ we denote by $\Ch(u(g))$ this composition. 
We have the following theorem from \cite{Benameur:thmFamilleLefschetz}.

\begin{thm}\cite{Benameur:thmFamilleLefschetz}\label{thm:benameur:nonTrivial}
The following equality is satisfied in 
$H(B^H,\mathbb{C})$ :
$$\dfrac{ \Ch\big(j^*\mathrm{Ind^{M|B}_H}(\sigma (h))\big)}{\Ch(\lambda_{-1}N^B(h))} =\int_{T^VM^H|B^H}\dfrac{\Ch\big(i^*\sigma(h)\big)\wedge \hat{A}^2(T^VM^H)}{\Ch\big(\lambda_{-1}(N\cap T^VM\otimes \mathbb{C} \oplus p^*N^B)(h)\big)},$$
where $\sigma (h)$, $\lambda_{-1}N^B(h)$ and $\lambda_{-1}(N\cap T^VM \otimes \mathbb{C} \oplus p^*N^B)(h)$ are the evaluation at the element $h$.
\end{thm}

Let $s\in H$ and $X\in \mathfrak{h}(s)$. Denote by $h=se^X$. Assume that $X$ is sufficiently close to zero such that the manifold $M^X$ and $M^{e^X}$ coincides, see \cite{BV:formuleloc:Kirillov,BGV,paradan2008index,DV:comoEquiDescente}.
On the one hand, if $m\in (M^s)^X$ then $s\cdot m=m$ and $X_M(m)=0$ so $e^X\cdot s\cdot m=m$ and therefore $h\cdot m=m$ because $e^Xs=se^X$. On the other hand, if $m\in M^h$ then $se^X\cdot m=m$ so $s_*X_M( m)=0$. Then we have $X_M(m)=0$ since $s_*$ is an isomorphism so $m\in M^X$, we deduce that $e^X\cdot m=m$. However $s\cdot m=se^X \cdot m=m$ therefore $m\in M^s$. We have $m\in (M^s)^X$ because $(M^s)^X=M^s\cap M^X$.

\begin{lem}\label{lem: Chlambda2}
Let $s\in H$ and $X\in \mathfrak{h}(s)$, we denote by $h$ the element $se^X$. Assume $X$ small enough so that $(M^s)^X=M^h$. We have the following inclusions \xymatrix{ 
M^h\ar@{^{(}->}[r]^{i_{h,s}} & M^s \ar@{^{(}->}[r]^{i_s} & M
} and we denote by $i_h$ the inclusion of $M^h$ in $M$. We denote by $N_h$ the normal bundle of $M^h$ in $M$, $N_h^s$ the normal bundle of $M^h$ in $M^s$ and $N_s$ the normal bundle of $M^s$ in $M$. Then we have:
\begin{subequations}
  \begin{alignat}{3}
\begin{split}	1.~& \hspace*{-0.05cm}{\Ch\big(\lambda_{-1}(N_h\hspace*{-0.08cm}\cap \hspace*{-0.08cm} T^V\hspace*{-0.08cm}M\otimes \mathbb{C})(se^X)\big) \hspace*{-0.08cm}}
	\\
	&{= \hspace*{-0.08cm}\Ch\big(\lambda_{-1}(N_h^s\hspace*{-0.08cm}\cap\hspace*{-0.08cm} T^V\hspace*{-0.08cm}M\otimes \mathbb{C}), X\big)i_{h,s}^*\Ch_s\big(\lambda_{-1}(N_s\hspace*{-0.08cm}\cap\hspace*{-0.08cm} T^V\hspace*{-0.08cm}M\otimes \mathbb{C}), X\big),}
	\label{lem : Chlambda2 : 1}
	\end{split}\\
	2.~&\hspace*{-0.05cm}\Ch\big(\lambda_{-1}(N_h^s\cap T^VM\otimes \mathbb{C}),X\big)=\Eul (N_h^s\cap T^VM\otimes \mathbb{C},X)\wedge \hat{A} (N^s_h\cap T^VM,X)^{-2},\label{lem :Chlambda2: 2}
	\\
	3.~&\hspace*{-0.05cm}\Ch_s\big(\lambda_{-1}(N_s\cap T^VM\otimes \mathbb{C}),X\big)=D_s(N_s\cap T^VM,X).
  \end{alignat}
\end{subequations}

\end{lem}

\begin{proof}~%
1. The element $se^X$ acts trivially on $M^h$. On the one hand, we have the following identifications: 
\begin{enumerate}
\item[-]$TM_{|M^s}=TM^s \oplus N_s$,
\item[-]$TM_{|M^h}=TM^h \oplus N_h=(TM^s)_{|M^h} \oplus (N_s)_{|M^h}=TM^h\oplus N_h^s\oplus (N_s)_{|M^h},$
\end{enumerate}
so $N_h\cap T^VM=N_h^s\cap T^VM \oplus (N_s)_{|M^h}\cap T^VM.$ 
We know that $\Lambda (E\oplus E')\cong\Lambda E \otimes \Lambda E'$ and hence
$$\begin{array}{lll}
\Ch\big(\lambda_{-1}(N_h\cap T^VM\otimes \mathbb{C})(se^X)\big)\\
=\Ch\big(\lambda_{-1}(N_h^s\cap T^VM\otimes \mathbb{C})(se^X)\big)\wedge\Ch\big(i_h^*\lambda_{-1}(N_s\cap T^VM\otimes \mathbb{C})(se^X)\big).\end{array}$$
On the other hand, the vector bundle $N_h^s$ is a vector subbundle of $TM^s$ on which $s$ acts trivially. So the action on $N^s_h$ is trivial. We deduce that the action of $s$ on $\lambda_{-1}(N_h^s\cap T^VM\otimes \mathbb{C})$ is trivial and we have:
$$\begin{array}{lll}
\Ch_{se^X}\big(\lambda_{-1}N_h^s\cap T^VM\otimes \mathbb{C}\big)&=\Ch\big(\lambda_{-1}N_h^s\cap T^VM\otimes \mathbb{C}(se^X)\big)\\
&=\Ch\big(\lambda_{-1}N_h^s\cap T^VM\otimes \mathbb{C},X\big) .\end{array}$$ 
Moreover, using Theorem \ref{thm:chern:paradan}, we have:
$$\begin{array}{lll}\Ch\big(i_h^*\lambda_{-1}(N_s\cap T^VM\otimes \mathbb{C})(se^X)\big)&=\Ch_{se^X}\big(\lambda_{-1}(N_s\cap T^VM\otimes \mathbb{C})\big)(0)\\
&=i_{h,s}^*\Ch_{s}\big(\lambda_{-1}(N_s\cap T^VM\otimes \mathbb{C}),X\big).\end{array}$$
2. Denote by $R_0(X)$ an equivariant curvature on $N^s_h \cap T^VM$ associated to a $H$-invariant metric connection. We have $\dfrac{e^{R_0/2}-e^{-R_0/2}}{R_0}=e^{-R_0/2}\dfrac{e^{R_0}-1}{R_0}$ and $\det\big(e^{-R_0/2}\big)=1$ since $R_0$ is antisymmetric. It follows that $\hat{A}(N^s_h\cap T^VM ,X)^{-2}=\det\bigg(\dfrac{-(1-e^{R_0(X)})}{R_0(X)}\bigg)$, so we get the result since $\Eul(N^s_h\cap T^VM \otimes \mathbb{C},X)=\Eul(N^s_h\cap T^VM,X)^2=\det \big(-R_0(X)\big)$.\\
3. Denote by $R(X)$ an equivariant curvature on $N^s\cap T^VM$ and $R(X)\otimes \mathbb{C}$ the equivariant curvature associated on $N^s\cap T^VM \otimes \mathbb{C}$. For any linear map $A$ on $\mathbb{R}^n$, if we denote by $\Lambda^{i}(A )$ the induced map on $\Lambda^i\mathbb{R}^n$ then we have:

\begin{equation}
 \sum (-1)^i\Tr\bigg(\Lambda^{i}\big(A \big)\bigg)=\det \big(1-A\big). \label{equ:Trlambda=det}
 \end{equation}
The result follows from (\ref{equ:Trlambda=det}) applied to $se^{R(X)\otimes \mathbb{C}}$.
\end{proof}

\begin{lem}\label{lem : Chlambda1:nonTrivial}  
In the cohomology with complex coefficients of $B^H$, the folowing equality is satisfied at the point $h=e^X$:
\begin{equation}\begin{split}
\Ch\big(\lambda_{-1}(N\cap T^VM &\otimes \mathbb{C}\oplus p^*N^B)(e^X)\big)\\
&=\Ch\big(\lambda_{-1}(N\cap T^VM \otimes \mathbb{C}),X\big)\wedge p^*\Ch (\lambda_{-1}N^B,X). \label{equation:Chlambda:nonTriviale}
\end{split}
	\end{equation}
\end{lem}

\begin{proof}
The proof stems from Lemma \ref{lem : Ch_g=Chchi} and the multiplicativity  of the equivariant Chern character.
\end{proof}

\begin{thm}
For $X\in \mathfrak{h}$ small enough, the following equality is satisfied in the equivariant cohomology group $H(B,d_X)$: 
$$\Ch\big(\mathrm{Ind^{M|B}_H}(\sigma),X\big)=\int_{T^VM|B}\Ch(\sigma,X)\wedge \hat{A}^2(T^VM,X).$$

\end{thm}

\begin{proof}By Theorem \ref{thm:benameur:nonTrivial}, we have the following equality in the cohomology of $B^H$: 
\begin{equation}
\dfrac{ \Ch\big(j^*\mathrm{Ind^{M|B}_H}(\sigma (e^X))\big)}{\Ch(\lambda_{-1}N^B(e^X))} =\int_{T^VM^H|B^H}\dfrac{\Ch\big(i^*\sigma(e^X)\big)\wedge \hat{A}^2(T^VM^H)}{\Ch\big(\lambda_{-1}(N\cap T^VM\otimes \mathbb{C} \oplus p^*N^B)(e^X)\big)}. \label{equation:benameur:nonTrivial:eX}
\end{equation}
We have $\Ch\big(j^*\mathrm{Ind^{M|B}_H}(\sigma (e^X))\big)=\Ch\big(j^*\mathrm{Ind^{M|B}_H}(\sigma),X\big)$ since the action is trivial on $B^H$.
By Lemma \ref{lem : Chlambda1:nonTrivial}, we get:
\begin{equation}\begin{split}
\Ch\big(\lambda_{-1}(N\cap T^VM &\otimes \mathbb{C}\oplus p^*N^B)(e^X)\big)\\
&=\Ch\big(\lambda_{-1}(N\cap T^VM \otimes \mathbb{C}),X\big)\wedge p^*\Ch (\lambda_{-1}N^B,X).
\end{split}
	\end{equation}
Now using Lemma \ref{lem: Chlambda2} equation \ref{lem :Chlambda2: 2} with $s=e$ and $h=e^X$, we obtain: 
$$\Ch\big(\lambda_{-1}(N\cap T^VM\otimes \mathbb{C}(e^X)\big)= \Eul(N\cap T^VM \otimes \mathbb{C},X)\wedge \hat{A}(N\cap T^VM ,X)^{-2}. $$
So we get the following equality:
\begin{align*}
  &\dfrac{ \Ch\big(j^*\mathrm{Ind^{M|B}_H}(\sigma ),X\big)}{\Ch(\lambda_{-1}N^B,X)}=\int_{T^VM^H|B^H}\dfrac{\Ch(i^*\sigma,X)\wedge \hat{A}^2(N\cap T^VM ,X)\wedge \hat{A}^2(T^VM^H)}{\Eul(N\cap T^VM \otimes \mathbb{C},X)\wedge p^*\Ch (\lambda_{-1}N^B,X)}.
\end{align*} 
Moreover, we have $\hat{A}^2(N\cap T^VM,X)\wedge\hat{A}^2(T^VM^H) =i^* \hat{A}^2(T^VM,X)$, so:
$$\dfrac{j^* \Ch\big(\mathrm{Ind^{M|B}_H}(\sigma ),X\big)}{\Ch(\lambda_{-1}N^B,X)}= \dfrac{1}{\Ch (\lambda_{-1}N^B,X)}\wedge \int_{T^VM^H|B^H}\hspace*{-0.5cm}\dfrac{i^*\big(\Ch(\sigma,X) \wedge \hat{A}^2(T^VM,X)\big)}{\Eul(N\cap T^VM \otimes \mathbb{C},X)}.$$
Then it follows that:
$$j^* \Ch\big(\mathrm{Ind^{M|B}_H}(\sigma ),X\big)= \int_{T^VM^H|B^H}\hspace*{-0.5cm}\dfrac{i^*\big(\Ch(\sigma,X) \wedge \hat{A}^2(T^VM,X)\big)}{\Eul(N\cap T^VM \otimes \mathbb{C},X)}.$$
Note that the normal bundle $\mathcal{N}(T^VM^X,T^VM)$ to the inclusion of $T^VM^X$ in $T^VM$ is isomorphic to $\pi^{*}(N^X_M\oplus N^X_M\cap T^VM)$ where $\pi : T^VM \rightarrow M$ is the projection. So $\mathcal{N}(T^VM^X,T^VM) \cap \ker(d(p\circ \pi))$ is isomorphic to $\pi^*\big((N^X_M\cap T^VM) \otimes \mathbb{C}\big)$.\\
Moreover, for $X \in \mathfrak{h}$ small enough, the zeros of the vector field generated by $X$ and the fixed points of $e^X$ coincide so 
$T^VM^H=T^VM^X$ and $B^H=B^X$.
Since $\Ch(\sigma , X)$ has compact support, we can apply the Bismut localization formula (Theorem \ref{thm : Bismut}), to get:
$$\int_{T^VM^X|B^X}\dfrac{i^*\big(\Ch(\sigma,X)\wedge \hat{A}^2(T^VM^H, X)\big)}{\Eul(N\cap T^VM \otimes \mathbb{C},X)}=j^*\int_{T^VM|B}\Ch(\sigma,X)\wedge \hat{A}^2(T^VM,X).$$
Now, using Proposition 2.1 of \cite{BV:formuleloc:Kirillov} which says that the restriction $j^* : H(B,d_X) \rightarrow H(B^X,\mathbb{C})$ is an isomorphism, we get the desired result.
\end{proof}

We now give a similar formula for the equivariant Chern character of the index of an elliptic family in the neighborhood of a point $ s \in H $ different from the identity. 

\begin{thm}\label{thm:BV:famille:elliptique} Let $s\in H$ and $X\in \mathfrak{h}(s)$, we denote by $h$ the element $se^X$. We suppose $X$ small enough, such that $(M^s)^X=M^h$. Denote by $N_s$ the normal bundle of $M^s$ in $M$.
The following equality is true in the equivariant cohomology $H(B^s,d_X)$: 
$$\Ch_{s}\big(\mathrm{Ind^{M|B}_H}(\sigma),X\big)=\int_{T^VM^s|B^s}\dfrac{\Ch_{s}(\sigma,X)\wedge \hat{A}^2(T^VM^s,X)}{D_s(N_s\cap T^VM,X)}.$$ 

\end{thm}

\begin{proof}We have the inclusions \xymatrix{ 
M^h\ar@{^{(}->}[r]^{i_{h,s}} & M^s \ar@{^{(}->}[r]^{i_s} & M
}, we denote by $i_h$ the inclusion of $M^h$ in $M$. Similarly, we have the inclutions \xymatrix{ 
B^h\ar@{^{(}->}[r]^{j_{h,s}} & B^s \ar@{^{(}->}[r]^{j_s} & B
}, we denote by $j_h$ the inclusion of $B^h$ in $B$. We denote by $N_h$ the normal bundle of $M^h$ in $M$, $N_h^s$ the normal bundle of $M^h$ in $M^s$.
By Theorem \ref{thm:benameur:nonTrivial} for $h$, on $B^h$, we get:
\begin{align*}
\dfrac{ \Ch\big(j_h^*\mathrm{Ind^{M|B}_H}(\sigma (h))\big)}{\Ch(\lambda_{-1}N^B(h))} =\int_{T^VM^h|B^h}\dfrac{\Ch\big(i_h^*\sigma(h)\big)\wedge \hat{A}^2(T^VM^H)}{\Ch\big(\lambda_{-1}(N_g\cap T^VM\otimes \mathbb{C} \oplus p^*N^B)(h)\big)}.
\end{align*}
By Theorem \ref{thm:chern:paradan} 1 and Lemma \ref{lem : Ch_g=Chchi}, we have the following equalities: 
\begin{enumerate}
\item[-]$\Ch\big(j_h^*\mathrm{Ind^{M|B}_H}(\sigma (h))\big)=j_{h,s}^*\Ch_s\big((\mathrm{Ind^{M|B}_H})(\sigma),X\big)$,
\item[-] $\Ch\big(i_h^*\sigma(h)\big)=i_{h,s}^*\Ch_s(\sigma,X).$
\end{enumerate}
In fact, let us spell the second equality the first is obtain similarly. We have $\Ch(j_h(\sigma)(se^X))=\Ch_{se^X}(\sigma,0)$ since the action of $h=se^X$ is trivial on $M^h$ and so now using the second part of assertion 1 of Theorem \ref{thm:chern:paradan} \cite[Theorem\  3.19]{paradan2008equivariant}, we know that $\Ch_{se^X}(\sigma,0)=\Ch_s(\sigma,X+0)_{|M^s\cap M^h}=j^{*}_{h,s}\Ch_s(\sigma ,X)$.\\ 
\noindent
Using Lemma \ref{lem : Chlambda1:nonTrivial}, applying (\ref{lem :Chlambda2: 2}) 
and the fact that $T^VM^h\oplus N^s_h =i^*_{h,s}T^VM^s$, we get:\\
$\begin{array}{lll} j_{h,s}^*\Ch_s\big(\hspace*{-0.3cm}&\mathrm{Ind^{M|B}_H}(\sigma),X\big)\\
&=\displaystyle\int_{T^VM^h|B^h} \dfrac{ i_{h,s}^*\big(\Ch_s(\sigma,X)\wedge \hat{A}^2(T^VM^s,X)\wedge D_s(N_s\cap T^VM,X)^{-1}\big)}{\Eul (N_h^s\cap T^VM\otimes \mathbb{C},X)}.\end{array}$\\
Now, for $X \in \mathfrak{h}(s)$ small enough, we have $M^h=(M^s)^X$ and $T^VM^h=(T^VM^s)^X$, so:\\
$\begin{array}{lll}
j_{h,s}^*\Ch_s\big(\hspace*{-0.3cm}&\mathrm{Ind^{M|B}_H}(\sigma),X\big)\\
&=\displaystyle\int_{(T^VM^s)^X|(B^s)^X} \hspace*{-0.6cm}\dfrac{ i_{h,s}^*\big(\Ch_s(\sigma,X)\wedge \hat{A}^2(T^VM^s,X)\wedge D_s(N_s\cap T^VM,X)^{-1}\big)}{\Eul (N_h^s\cap T^VM\otimes \mathbb{C},X)},
\end{array}$\\
by Theorem \ref{thm : Bismut}, it follows that:\\
$\begin{array}{lll}
j_{h,s}^*\Ch_s\big(\hspace*{-0.3cm}&\mathrm{Ind^{M|B}_H}(\sigma),X\big)\\
&=j_{h,s}^*\int_{T^VM^s|B^s} \Ch_s(\sigma,X)\wedge \hat{A}^2(T^VM^s,X)\wedge D_s(N_s\cap T^VM,X)^{-1}.\end{array}$\\
Moreover, the restriction $j_{h,s}^* : H(B^s,d_X) \rightarrow H(B^h,\mathbb{C})$ is an isomorphism, so we conclude:
$$\Ch_{s}\big(\mathrm{Ind^{M|B}_H}(\sigma),X\big)=\int_{T^VM^s|B^s}\dfrac{\Ch_{s}(\sigma,X)\wedge \hat{A}^2(T^VM^s,X)}{D_s(N_s\cap T^VM,X)}.$$ 
\end{proof}

\begin{remarque} \cite{Atiyah-Segal:II}\label{rem:remplacerGroupe}
To remove the hypothesis that $H$ is a topologically cyclic group generated by a single element, it is enough to replace $H$ by the closure of the group generated by the element $h$. Indeed, if we denote by $\varphi : \tilde{H}=\overline{<h>} \hookrightarrow H$ the inclusion of the closure of the group generated by $h$ in $H$ then the following diagram is commutative: 
$$\xymatrix{ \Khh(T^VM) \ar[r]^{\varphi^*} \ar[d]_{\mathrm{Ind^{M|B}_H}} & \mathrm{K}_{\tilde{\mathrm{H}}}(T^VM) \ar[d]^{\mathrm{Ind}_{\tilde{\mathrm{H}}}^{\mathrm{M}|\mathrm{B}}} \\
\Khh(B) \ar[r]_{\varphi^*} & \mathrm{K}_{\tilde{\mathrm{H}}}(B).}$$
\end{remarque}

\begin{remarque}
If the action of $H$ is trivial on $B$ then we obtain equalities in the de Rham cohomology $H(B,\mathbb{C})$ of $B$ with complex coefficient.
\end{remarque}

\subsection{An application: The homogeneous case}
The purpose of this section is to explain the link between the Berline-Vergne delocalization formula for a $G \times H $-invariant elliptic operator and the formula along the fibers presented in the previous section. We begin by recalling a construction from \cite{Atiyah-Singer:I}. \\
Let $G$ and $H$ be two compact Lie groups. Let $q :P\rightarrow B $ be a  $H$-equivariant $G$-principal bundle. Let $F$ be a $G\times H $-manifold. We define a $H$-equivariant fibration $p:M\rightarrow B$, with fiber $F$ and structural group $G$, by setting $M=P \times_G F $. Let $A:C^{\infty}(F,E^+)\rightarrow C^{\infty}(F,E^-)$ be an elliptic pseudodifferential $G \times H$-invariant operator of order $1$. By \cite{Atiyah-Singer:I}, we know that the index $\mathrm {Ind}_{G\times H} (A)$ of the operator $A$ is an element of $R(G\times H)$. Recall \cite[(4.3) page 504]{Atiyah-Singer:I} the map
$$ \mu_P: R (G \times H) \rightarrow \Khh (B) $$
induced by the map which associates to a $G \times H$-representation $ V $ the vector bundle over $B$ given by $P \times_G V$. We denote respectively $p_1$ and $p_2$ the first and second projections of $P \times F $. Following \cite[page 527]{Atiyah-Singer:I}, we define a $H$-invariant operator  $\tilde{A} $ on $M$, elliptic along the fibers. The operator $ A $ lifts to an operator $\tilde{A}_1 $ on $ P\times F $. Since $ \tilde{A} _1 $ is $ G \times H $-invariant, it induces an operator $ \tilde{A} $ on $M$. We restrict $ \tilde{A}_1 $ to the constant sections along the fibers of $ P\times F \rightarrow M $. Since $ P $ is locally a product, the restriction $ \tilde{A}_V $ of $ \tilde{A}$ to the open sets $ p^{-1}(V) $ is just the lift of $A$, so $\tilde{A}_V\in \overline{\mathcal{P}}^1(p^{-1}(V))$ and therefore $\tilde{A}\in \overline{\mathcal{P}}^1(M)$. The symbol $\sigma(\tilde{A})$ satisfies $\sigma_{(\eta,\xi)}(\tilde{A})=\sigma_{\xi}(A)$ so is elliptic along the fibers of $p$. Moreover, we have the following proposition:

\begin{prop}[\cite{Atiyah-Singer:I}, page 529]
The index of $\tilde{A}$ is given by:
$$\begin{array}{lll}
\mathrm{Ind^{M|B}_H}(\tilde{A})&=[P\times_G \ker A] -[P\times_G \mathrm{coker }~A]\in \Khh(B)\\
&=\mu_P(\mathrm{Ind}^F_{G\times H}(A)).
\end{array}$$
\end{prop}

Let $s\in H$. We denote again by $\mathfrak{h}$ the Lie algebra of $H$ and by $\mathfrak{h}(s)$ the Lie algebra of the centralizer $H(s)$ of $s$ in $H$.

\begin{cor}\label{cor:Ch(indm)=Ch(muPind)}
Let $X\in \mathfrak{h}(s)$. We have the following equality in $H(B^s,d_X)$:
$$\Ch_s(\mathrm{Ind^{M|B}_H}(\tilde{A}),X)=\Ch_s( \mu_P(\mathrm{Ind}^F_{G\times H}(A)),X).$$
\end{cor}

In the following, $CW_\g $ will equally denote the Chern-Weil homomorphism $C^{\infty}(\mathfrak{h} \times \g )^{H\times G} \rightarrow \mathcal{H}^{\infty}_H(\mathfrak{h} ,B)$ and the Chern-Weil isomorphism $\mathcal{H}^{\infty}_{H\times G}( \mathfrak{h} \oplus \g, P)\rightarrow \mathcal{H}^{\infty}_H(\mathfrak{h}, B)$. We denote by $ \theta \in (\mathcal{A}^{1}(P)\otimes \g)^{G \times H}$ a $H$-invariant $1$-form connection on $P$ and $ \Theta=d_{\mathfrak{h}}\theta +\frac{1}{2}[\theta,\theta] $ its curvature. We denote by $ \Theta(X) = \Theta - \iota(X) \theta \in \mathcal{A}^{\infty}_H(\mathfrak{h},P)\otimes \mathfrak{g} $ its equivariant curvature. Let $(X_j)$ be a basis of $\g$ and denote by $(x_j)$ its dual basis. 
Write $\theta =\sum \limits_j\theta^j\otimes X_j$ and $\Theta =\sum \limits \Theta^j \otimes X_j$. Let $h= \prod \limits_j (I-\theta^j\otimes \iota(X_j)) : \mathcal{A}^\infty_{H}(\mathfrak{h},P) \rightarrow \mathcal{A}^{\infty}_H(\mathfrak{h},B) $ be the horizontal projection. If $\phi \in C^\infty(\g)$ then using the Taylor formula we define $\phi(\Theta)$ by
$$\phi(\Theta)(X)=\sum \limits_{J\subset \{1,\cdots ,\dim \g\}}\dfrac{\Theta^J}{J!}(\partial_J\phi)(\iota(X)\theta),\ X\in \mathfrak{h},$$
where $\Theta^J=\prod\limits_{j\in J}\Theta^j$ and $\partial_J=\prod \limits_{j\in J}\dfrac{\partial}{\partial x^j}$.
The Chern-Weil isomorphism is given by $CW_\g(\alpha )(X)=h(\alpha (X,\Theta(X)))$, $\forall \alpha \in \mathcal{H}^{\infty}_{H\times G}(\mathfrak{h}\oplus \g, P)$ and $X\in \mathfrak{h}$. For more details, see \cite{DV:comoEquiDescente} (see also \cite{Guillemin:Sternberg:SupersymEquiv,meinrenken2006equivariant}). We have:

\begin{prop}\label{prop:Chern-Weil}
The following diagram is commutative:
$$\xymatrix{ R(H(s)\times G) \ar[r]^{\mu_P} \ar[d]_{\Ch_s=\chi_{_-}(se^{- })} & \k_{\mathrm{H(s)}}(B^s)\ar[d]^{\Ch_s}\\
C^{\infty}(\mathfrak{h}(s) \times \mathfrak{g})^{H(s)\times G} \ar[r]_{CW} & \mathcal{H}^{\infty}_{H(s)}(\mathfrak{h}(s),B^s).
} $$

\end{prop}

Applying this proposition and the Berline-Vergne formula \cite{BV:formuleloc:Kirillov}, we get:

\begin{cor}\label{cor:Ch(indm)parBV}
Let $X\in \mathfrak{h}(s)$ small enough. We have the following equality in $H(B^s,d_X)$:
$$\Ch_s(\mathrm{Ind^{M|B}_H}(\tilde{A}),X)=\int_{TF^s} \dfrac{\Ch_s(\sigma(A),X,\Theta (X))\wedge \hat{A}^2(TF^s,X,\Theta(X))}{D(\mathcal{N}(F^s,F),X,\Theta(X))},$$
where $\Ch_{(s,e)}(\sigma(A),X,Y) \in \mathcal{H}^\infty_{H(s)\times G}(\mathfrak{h}(s)\times \g,TF)$ is the $s$-equivariant Chern character of $\sigma(A)$.
\end{cor}

\begin{proof}
By Corollary \ref{cor:Ch(indm)=Ch(muPind)}, we know that 
$$\Ch_s(\mathrm{Ind^{M|B}_H}(\tilde{A}),X)=\Ch_s( \mu_P(\mathrm{Ind}^F_{G\times H}(A)),X).$$
By Proposition \ref{prop:Chern-Weil}, it follows that:
$$\Ch_s(\mathrm{Ind^{M|B}_H}(\tilde{A}),X)=CW\big(\mathrm{Ind}^F_{G\times H}(A)\big)(X).$$
So we get the result by applying the Berline-Vergne formula \cite{BV:formuleloc:Kirillov} to $\mathrm{Ind}^F_{G\times H}(A)$.

\end{proof}

Let $s\in H$. We have $T^VM = P\times_G TF$ and $T^VM^s=q^{-1}(B^s)\times_G TF^s)$ because the action of $H$ commutes with the action of $G$. Denote by $\mathcal{N}(F^s,F)$ the normal bundle of $F^s$ in $F$. The vertical part of the normal bundle $M^s=q^{-1}(B^s)\times_G F^s$ in $M$ is given by $\mathcal{N}(M^s,M)\cap T^VM=q^{-1}(B^s)\times_G \mathcal{N}(F^s,F)$. 
Denote by $p_1 : q^{-1}(B^s)\times F^s \rightarrow q^{-1}(B^s)$ and $p_2 : q^{-1}(B^s)\times F^s \rightarrow F^s$ the projections. The $1$-form $p_1^*\theta$ is a connection on $P\times F \rightarrow P\times_GF$ which restrict to a connection on $q^{-1}(B^s)\times F^s\rightarrow q^{-1}(B^s)\times_GF^s$. We denote by $CW_{q^{-1}(B^s)\times F^s}$ the Chern-Weil isomorphism associated to the bundle $q^{-1}(B^s)\times H ^s \rightarrow q^{-1}(B^s)\times_G F^s$. We have:
\noindent
\begin{prop}\label{prop:egalite:classes:A-genre:D:Chern} We have the following equalities:
\begin{enumerate}

\item $\hat{A}(q^{-1}(B^s)\times_G TF^s,X)=p_2^*\hat{A}(TF^s,X,p^*_1\Theta(X))$;
\item $D(\mathcal{N}(M^s,M)\cap T^VM,X)=p_2^*D(\mathcal{N}(F^s,F),X,p_1^*\Theta(X))$;
\item $\Ch_s(\sigma(\tilde{A}),X)=p_2^*\Ch_s(\sigma(A),X,p_1^*\Theta(X))$.
\end{enumerate}
\end{prop} 

\begin{proof}
We only give the details of the proof of the first equality. Assertions $2$ and $3$ can be shown in the same way.\\
We have:
$$\hat{A}(q^{-1}(B^s)\times_G TF^s,X)=CW_{q^{-1}(B^s)\times F^s}\big(\hat{A}(q^{-1}(B^s)\times TF^s,-,-)\big)(X).$$
As $q^{-1}(B^s)\times TF^s$ is the pullback by $p_2$ of $TF^s$, we obtain:
$$\hat{A}(q^{-1}(B^s)\times_G TF^s,X)=CW_{q^{-1}(B^s)\times F^s}\big(p_2^*\hat{A}( TF^s,-,-)\big)(X).$$
So we get the result by applying the Chern-Weil isomorphism with the equivariant curvature $p^*_1\Theta(X)$.
\end{proof}

We will verify that the formula of index for families coincides with the formula obtained in Corollary \ref{cor:Ch(indm)parBV} by a direct calculation.

\begin{cor}
Let $X\in \mathfrak{h}(s)$ small enough. We have the following equality in $H(B^s,d_X)$:
$$\Ch_s(\mathrm{Ind^{M|B}_H}(\tilde{A}),X)=\int_{TF^s} \dfrac{\Ch_s(\sigma(A),X,\Theta (X))\wedge \hat{A}^2(TF^s,X,\Theta(X))}{D(\mathcal{N}(F^s,F),X,\Theta(X))}.$$
\end{cor} 

\begin{proof}
We start by applying Theorem \ref{thm:BV:famille:elliptique} to compute the Chern character of $\mathrm{Ind^{M|B}_H}(\tilde{A})$. We have:
$$\Ch_s(\mathrm{Ind^{M|B}_H}(\tilde{A}),X)=\int_{q^{-1}(B^s)\times_G TF^s|B^s} \dfrac{\Ch_s(\sigma(\tilde{A}),X)\wedge \hat{A}^2(q^{-1}(B^s)\times_G TF^s,X)}{D(q^{-1}(B^s)\times_G \mathcal{N}(F^s,F),X)}.$$
By Proposition \ref{prop:egalite:classes:A-genre:D:Chern}, we get:
$$\Ch_s(\mathrm{Ind^{M|B}_H}(\tilde{A}),X)=\int_{q^{-1}(B^s)\times_G TF^s|B^s} \hspace*{-0.8cm}\dfrac{p^*_2\Ch_s(\sigma(A),X,p_1^*\Theta(X))\wedge p^*_2\hat{A}^2(TF^s,X,p^*_1\Theta(X))}{p_2^*D(\mathcal{N}(F^s,F),X,p^*_1\Theta(X))}.$$
Denoting by $k : q^{-1}(B^s)\times TF^s \rightarrow q^{-1}(B^s)\times_G TF^s$ the projection, it follows that:\\
$\begin{array}{lll}
\Ch_s(\hspace*{-0.3cm}&\mathrm{Ind^{M|B}_H}(\tilde{A}),X)\\
&=\displaystyle \int_{q^{-1}(B^s)\times TF^s|B^s} k^*\bigg(\dfrac{p^*_2\Ch_s(\sigma(A),X,p_1^*\Theta(X))\wedge p^*_2\hat{A}^2(TF^s,X,p^*_1\Theta(X))}{p_2^*D(\mathcal{N}(F^s,F),X,p^*_1\Theta(X))}\bigg)\\
&=\displaystyle \int_{TF^s} \dfrac{\Ch_s(\sigma(A),X,\Theta(X))\wedge \hat{A}^2(TF^s,X,\Theta(X))}{D(\mathcal{N}(F^s,F),X,\Theta(X))}.
\end{array}
$
\end{proof}

\section{Berline-Vergne formula for a $G$-transversally elliptic family}

This section is an application of the cohomological formula of section \ref{Chern:caracter:index:class}. We use the Berline-Vergne formula \cite{BV:IndEquiTransversal} and more precisely the Paradan-Vergne approach from \cite{paradan2008index} to show a similar result for families. We assume in this section that $G$ acts trivially on $B$.
\subsection{Vertical deformation of the Chern character}

The Liouville $1$-form on $T^*M$ allows to define by restriction a "vertical Liouville $1$-form".
More precisely, let us fix a riemannian metric $\langle \cdot , \cdot \rangle$ on $M$. Then we can write $TM=T^VM\bigoplus p^*TB$.
Let $\pi : T^*M \rightarrow M $ be the projection, let $j : T^*M \rightarrow T^VM^*$ be the dual map to the inclusion $i : T^VM\hookrightarrow TM$ and let $\phi : T^*M \rightarrow TM$ be the isomorphism given by the metric on $M$ and let $\phi^V : T^VM^* \rightarrow T^VM$ be the induced isomorphism. Denote by $r=\phi^{-1}\circ i \circ \phi^V$ and $k=r \circ j$. The Liouville $1$-form $\omega$ on $T^*M$ is the $1$-form defined by $\langle \omega(\xi ), \zeta \rangle = \langle \xi ,d\pi(\zeta)\rangle$, where $\xi \in T^*M$, $\zeta \in T_\xi(T^*M)$ and $\pi : T^*M \rightarrow M$ is the projection. As before let $G$ be a compact Lie group and $p :M\rightarrow B$ a $G$-equivariant fibration of compact manifolds. Assume $B$ oriented.

\begin{lem}Let $\omega $ be the Liouville $1$-form on $T^*M$. The $1$-form $k^*\omega $ is $G$-invariant and the subspace $C_{k^*\omega}$ of $T^*M$ is equal to $C_\omega =T^*_GM$ (see Section \ref{section:Chern:Character} for the definition of $C_\omega$). 
\end{lem}

\begin{proof}
The $1$-form $k^*\omega $ is $G$-invariant because $k$ is $G$-equivariant. Let $\xi \in T^*M$ and $v\in T_\xi (T^*M)$. We have:
$$
\langle (k^*\omega)_\xi , v \rangle = \langle \omega_{k(\xi)} , d_\xi k(v) \rangle = \langle k(\xi) , d_{k(\xi)}\pi \circ d_\xi k(v) \rangle .
$$
Furthermore, we have the equality $\pi \circ k =\pi$ so we get:
$$ \langle k^*\omega , v \rangle = \langle k(\xi) , d\pi (v) \rangle . $$
Now if $v=X^{*}_{T^*M}(\xi)$ is given by an element $X\in \g$ then $d\pi(v)=X^{*}_M(\pi(\xi))$ is a vertical vector, that is $d\pi(v) \in T_{\pi(\xi )}^VM$ so $\langle k^*\omega , v \rangle = \langle \omega , X^{*}_{T^*M}(\xi) \rangle $ since the horizontal part of $\xi $ vanish on the vertical vectors. So we get that $C_{k^*\omega}$ is equal to $C_\omega$.
\end{proof}

\begin{cor}\label{cor:egalite:chern}
The Chern character in $\mathcal{H}^{-\infty}_G(\g ,T^*M)$ defined using the Liouville $1$-form $\omega$ is equal to the Chern character defined using the vertical Liouville $1$-form $k^*\omega $. 
\end{cor}

\begin{proof}
By Theorem 3.19 of \cite{paradan2008equivariant} (see also Theorem \ref{thm:chern:paradan}), we know that if $\sigma_\tau :E^+ \rightarrow E^-$ is a family of smooth $G$-invariant morphisms and $\lambda_\tau$ a family of $G$-invariant $1$-forms such that $C_{\lambda_\tau,\sigma_\tau} \subset F$, for $\tau\in [0,1]$ and $F$ a closed subspace of $T^*M$, then all the classes $\Ch_{\sup } (\sigma_\tau,\lambda_\tau)$ coincide in $\mathcal{H}_F^{-\infty}(\mathfrak{g},N)$. We take for $\sigma_\tau$ the constant family of  morphisms and for $\lambda_\tau$ the family of $1$-forms $\tau\omega +(1-\tau)k^*\omega $. Note that the family $\lambda_\tau$ is $G$-equivariant. We have for $v\in T_\xi(T^*M)$:
$$\langle \lambda_\tau ,v \rangle = \tau \langle \xi  ,d\pi(v) \rangle +(1-\tau)\langle k(\xi ) ,d\pi(v) \rangle $$
so if $v$ is given by an element $X\in \g$ then $\langle \lambda_\tau ,v \rangle =0$ if $\xi \in T^*_GM$. We get that $C_{\lambda_\tau}=C_\omega =C_{k^*\omega }$ which completes the proof. 
\end{proof}

\begin{lem}
Let $\omega $ be the Liouville $1$-form on $T^*M$. The $1$-form $r^*\omega $ is $G$-invariant and the subspace $C_{r^*\omega }$ of $T^VM^*$ is equal to $T^V_GM^*$.

\end{lem}

\begin{remarque}
The $1$-form $r^*\omega$ is the restriction of $\omega $ to $T^VM^*$ that we see as a submanifold of $T^*M$ using the metric.
\end{remarque}

\begin{proof}
The $1$-form $r^*\omega $ is $G$-invariant because $r$ is $G$-equivariant.
Let $\xi \in T^VM^*$. Let $v\in T_\xi(T^VM^*)$. We have:
$$
\langle r^*\omega , v \rangle = \langle \omega_{r(\xi)} , dk(v) \rangle = \langle r(\xi) , d\pi \circ dr(v) \rangle =\langle \xi , d\pi(v) \rangle.
$$
because $r$ is the inclusion of $T^VM^*$ in $T^*M$ and $\pi \circ r$ is the projection $\pi$ restricted to $T^VM^*$. From this we deduce that the map $f_{r^*\omega } : T^VM^* \rightarrow \g^*$ is zero if $\xi \in T^V_GM^*$.
\end{proof}

\subsection{Berline-Vergne formula}

We begin by recalling the cohomological formula from \cite{paradan2008index} for a $G$-transversally elliptic operator. We will deduce from this formula, with the help of a Kasparov product, the formula for a family of $G$-transversally elliptic operators. We denote by $\omega_s $ the Liouville $1$-form on $T^*M^s$ and we use the notation of the previous section. In this section, we use evaluation of generalized functions, see Section \ref{section:equi:coh:coefgen} and \cite{paradan2008index,kumar1993equivariant} for more details.

\begin{thm}[\cite{paradan2008index}, Theorem 3.18]\label{thm:Ind:BV}
Let $\sigma$ be a symbol of $G$-transversally elliptic operator on a compact $G$-manifold $M$. Denote for any  $s\in G $, by $N^s$ the normal vector bundle to $M^s$ in $M$. There is a unique $G$-invariant generalized function on $G$ denoted $\mathrm{Ind}^{G,M}_{coh}([\sigma ])$, such that the following local relations are satisfied: 
$$\mathrm{Ind}^{G,M}_{coh}([\sigma ])\|_s(Y)
=(2i\pi)^{-\dim M^s}\bigint_{T^*M^s}\dfrac{\Ch_{c}(\sigma,\omega,s)(Y)\wedge \hat{A}^2(TM^s,Y)}{D_s(N^s,Y)},$$
$\forall s\in G$ and $\forall Y \in\g(s)$ small enough so that the equivariant classes $\hat{A}^2(TM^s,Y)$ and $D(N^s,Y)$ are defined. 
Moreover, the generalized function $\mathrm{Ind}^{G,M}_{coh}([\sigma ])$ only depends on the class of $\sigma$ in $\K(T^*_GM)$. 
\end{thm}

Furthermore, we have the following theorem which makes the link with the analytical index $\mathrm{Ind}^{G,M}_a$ of Atiyah \cite{atiyah1974elliptic} which associate to a $G$-transversally elliptic operator $P$ the following distributional character:
$$\mathrm{Ind}^{G,M}_a(P)=\chi_{\ker P} - \chi_{\ker P^*}.$$
\begin{thm}[\cite{paradan2008index}, Theorem 4.1]\label{thm:ind_c=ind_a}
The previous formulas define a map $$\mathrm{Ind}^{G,M}_{coh} : \K(T_G^*M) \rightarrow C^{-\infty}(G)^{Ad(G)}$$
and we have:
$$\mathrm{Ind}^{G,M}_a=\mathrm{Ind}^{G,M}_{coh}.$$

\end{thm}

\begin{lem}  Denote by $j :T^*M \rightarrow T^VM^*$ the projection, $r :T^VM^* \hookrightarrow T^*M$ the inclusion induced by the metric and $p : M\rightarrow B$ the projection.
Let $\sigma \in \K(T^V_GM^*)$ and $\sigma' \in \k(T^*B)$. We have the following equality in $\mathcal{H}^{-\infty}_G(\g ,T^*M^s)$:
$$\Ch_{c}(\sigma \otimes p^*\sigma',\omega,s)(Y)=j^*\Ch_{c}(\sigma,r^*\omega,s)(Y)\wedge p^*\Ch_{c}(\sigma').$$
\end{lem}
Denote by $E$ the super-bundle corresponding to $\sigma $ and $E'$ the super-bundle corresponding to $\sigma'$. In this lemma, $\mathbb{A}^{\omega_s}$ means the restriction of a super-connection $\mathbb{A}^{\omega}$ on $E\otimes p^*E'$, $\mathbb{A}^{r^*\omega_s}$ means the restriction of $\mathbb{A}^{r^*\omega}(\sigma)= \mathbb{A}+i(v_\sigma + r^*\omega)$, where $\mathbb{A}$ is a super-connection on $E$ and $\mathbb{A}(\sigma')=\mathbb{A'}+iv_{\sigma'}$ where $\mathbb{A}'$ is a super-connection on $E'$.

\begin{proof}
By Corollary \ref{cor:egalite:chern}, we have the equality:
$$\Ch_{c}(\sigma \otimes p^*\sigma' ,\omega,s)(Y)=\Ch_{c}(\sigma \otimes p^*\sigma' ,j^*r^*\omega ,s)(Y).$$
Furthermore, if we consider the product super-connection
$$\mathbb{B}=j^*\mathbb{A}\otimes 1 + 1\otimes p^*\mathbb{A}',$$
then we get:
$$\mathbb{A}^{j^*r^*\omega}(\sigma \otimes p^*\sigma' )=j^*(\mathbb{A}^{r^*\omega}(\sigma ) \otimes 1)+1\otimes p^*(\mathbb{A}'(\sigma')),$$
because $v_{\sigma\otimes p^*\sigma'}=v_{\sigma}\otimes 1 +1\otimes p^*v_{\sigma'}$. So the equality.
\end{proof}

Let now $\sigma$ be a $G$-transversally elliptic symbol along the fibers of $p : M\rightarrow B $. We defined in Section \ref{chapitre 2} the Chern character of the index class $\Ch^{\mathrm{HL}}(\ind([\sigma ]))$. Moreover, we identified it with $\Ch(\ind([\sigma ]))\in C^{-\infty }(G,H^{2\bullet}(B,\mathbb{C}))^{Ad(G)}$ in the Theorem \ref{thm:convergence:chern:indice}. We can restrict such element through its associated generalized function because such element is an element of $C^{-\infty}(G)^{Ad(G)}\otimes H(B,\mathbb{C})$.

We can now state our main theorem:

\begin{thm}\label{thm:BV:familles}
Let $\sigma$ be a $G$-transversally elliptic symbol along the fibers of a compact $G$-equivariant fibration $p : M\rightarrow B $ with $B$ oriented and $G$-trivial. Denote by $N^s$ the normal vector bundle to $M^s$ in $M$. \\
1. There is a unique generalized function with values in the cohomology of $B$ denoted $\mathrm{Ind}^{G,M|B}_{coh} : \K(T^V_GM) \rightarrow C^{-\infty}(G , H(B,\mathbb{C}))^{Ad(G)}$ satisfying the following local relations:
$$\mathrm{Ind}^{G,M|B}_{coh}([\sigma ] )\|_s(Y)=(2i\pi)^{-\dim (M^s|B)}     \bigint_{T^VM^s|B}    \hspace*{-0.5cm} \dfrac{\Ch_{c}(\sigma ,r^*\omega, s)(Y)\wedge \hat{A}^2(T^VM^s,Y)}{D_s(N^s,Y)},$$
$\forall s\in G$, $\forall Y \in \g(s)$ small enough such that the equivariant classes $\hat{A}^2(T^VM^s,Y)$ and $D(N^s,Y)$ are defined. \\
2. Furthermore, we have the following index formula:
$$\mathrm{Ind}^{G,M|B}_{coh}([\sigma ] )=\Ch(\ind([\sigma ]))\in C^{-\infty}(G ,H(B,\mathbb{C}))^{Ad(G)}.$$
\end{thm}

\begin{proof} Recall that $G(s)$ means the centralizer of $s$ in $G$ and that $\g(s)$ is its Lie algebra.
Denote by $U_s(0)\subset \g(s)$ a $G(s)$-invariant neighborhood of $0$, such that the $\hat{A}$-genus $\hat{A}(T^VM^s,Y)$ as well as $\dfrac{1}{D_s(N^s,Y)}$ are defined on $U_s(0)$. Denote by $n_s$ the dimension of $M^s$ and $n$ the dimension of $B$. By Theorem \ref{thm:convergence:chern:indice}, we have $\Ch(\ind([\sigma ] ))\|_s(Y) \in \mathcal{H}^{-\infty}_{G(s)}(U_s(0),B)\cong C^{-\infty}(U_s(0))^{Ad(G(s))} \otimes H(B,\mathbb{C})$. To compute $\Ch(\ind([\sigma ] ))\|_s(Y)$ it is sufficient to pair it with the de Rham homology of $B$. Moreover, the de Rham homology of $B$ is generated by the range of the Chern character of the $\k$-homology of $B$. By Corollary \ref{cor:couplage:ind:Atiyah}, we know that the pairing of the index class of a family of $G$-invariant operators which are $G$-transversally elliptic with an element of the $\k$-homology of $B$ is represented by a $G$-invariant, $G$-transversally elliptic operator on $M$. Let $\sigma'$ be a symbol on $B$ then we can apply Theorem \ref{thm:Ind:BV} to the $G$-transversally elliptic symbol  $\sigma\otimes p^*\sigma'$, 
to deduce (\cite{paradan2008index}):
$$ \mathrm{Ind}^{G,M}_{coh}([\sigma\odot p^*\sigma' ])\|_s(Y)
=(2i\pi)^{-n_s}\bigint_{T^*M^s}\hspace*{-0.2cm}\dfrac{\Ch_{c}(\sigma \odot p^*\sigma', \omega,s)(Y)\wedge \hat{A}^2(TM^s,Y)}{D_s(N^s,Y)}.$$
By Theorem \ref{thm:bismut:fibréNormal} we know that $TM^s\cong T^VM^s\oplus p^*TB$ and that $N$ is vertical because the action of $G$ on $B$ is trivial. So we have: 
$$\hat{A}(TM^s,Y)=\hat{A}(T^VM^s,Y)\wedge p^*\hat{A}(TB).$$
Moreover, we have $\Ch_{c}(\sigma \odot p^*\sigma' ,\omega,s)(Y)=j^*\Ch_{c}(\sigma, r^*\omega,s)(Y)\wedge p^*\Ch_c(\sigma' )$, where $j : T^*M \rightarrow T^VM^*$. 
So we get\\
$\begin{array}{lll}
\mathrm{Ind}^{G,M}_{coh}([\sigma\odot p^*\sigma' ])(se^Y)\\
\displaystyle=(2i\pi)^{-n_s}\bigint_{T^*M^s}\dfrac{j^*\Ch_{c}(\sigma, r^*\omega,s) (Y) \hat{A}^2(T^VM^s,Y)}{D_s(N,Y)} p^*\big(\Ch_c(\sigma') \hat{A}^2(TB)\big).
\end{array}$\\
As $B$ is oriented, we get:\\
$\begin{array}{lll}
\mathrm{Ind}^{G,M}_{coh}([\sigma\odot p^*\sigma' ])(se^Y)\\
=\displaystyle\bigint_{\hspace*{-0.1cm} b\in B}\nu\bigint_{\hspace*{-0.1cm} (TM_b)^s}\hspace*{-0.7cm}\dfrac{\Ch_{c}(\sigma , r^*\omega,s) (Y) \hat{A}^2(T^VM^s,Y)}{D_s(N^s,Y)} \hspace*{-0.1cm} \bigint_{\hspace*{-0.1cm} T_bB} \hspace*{-0.3cm}(2i\pi)^{-n} \Ch_c(\sigma' ) \hat{A}^2(TB),
\end{array}$\\
where $\nu=(2i\pi)^{-(n_s-n)}$. But, $\bigint_{\hspace*{-0.1cm} TB|B} \hspace*{-0.3cm}(2i\pi)^{-n} \Ch_c(\sigma' ) \hat{A}^2(TB)=PD(\mathrm{Ch}([P_{\sigma'}]))$ where $[P_{\sigma'}]\in \k\k(C(B) ,\mathbb{C})$ and where $\mathrm{Ch}([P_{\sigma'}])$ is the Chern character of $[P_{\sigma'}]$ in de Rham homology of $B$ and $PD : H_{\dim B -\bullet}(B,\mathbb{C})\rightarrow H^\bullet (B,\mathbb{C})$ is the Poincaré duality isomorphism. So we get that
$$\mathrm{Ind}^{G,M|B}_{coh}([\sigma ] )\|_s(Y)=(2i\pi)^{-\dim (M^s|B)}\hspace*{-0.1cm}\bigint_{T^VM^s|B}\hspace*{-0.5cm}\dfrac{\Ch_{c}(\sigma,r^*\omega,s)(Y)\wedge \hat{A}^2(T^VM^s,Y)}{D_s(N^s,Y)},$$
defines an element of $C^{-\infty}(G,H(B,\mathbb{C}))^{Ad(G)}=C^{-\infty}(G)^{Ad(G)}\otimes H(B,\mathbb{C})$, since $\forall \sigma' \in \k(TB)$ the generalized functions $(\mathrm{Ind}^{G,M}_{coh}(\sigma \otimes \sigma')\|_s )_s$ defines an element of $C^{-\infty}(G)^{Ad(G)}$ by Theorem 3.18 of \cite{paradan2008index}.\\
By Theorem 
4.5 of \cite{baldare:KK}, we know that $$\mathrm{Ind}^{G,M}_a([\sigma \odot p^*\sigma' ])=\ind ([\sigma ])\otimes_{C(B)}[P_{\sigma'}].$$
We denote by $\Ch^{\mathrm{HL}}$ the Chern character in bivariant local cyclic homology \cite{puschnigg2003diffeotopy}. 
The following diagram is commutative:
$$ \xymatrix{C^{-\infty}(G,H(B,\mathbb{C}))^{Ad(G)}\otimes H_*(B,\mathbb{C})\ar[r]^{\hspace*{1.8cm}\langle \cdot ,\cdot \rangle } \ar@{^{(}->}[d]&C^{-\infty}(G)^{Ad(G)}\ar@{^{(}->}[d]\\
\mathrm{Hom}(R(G)\otimes \mathbb{C} ,H^*(B,\mathbb{C}))\otimes H_*(B,\mathbb{C})\ar[d]_{\cong} \ar[r]^{\hspace*{1.8cm}\langle \cdot , \cdot \rangle }& \mathrm{Hom}(R(G)\otimes \mathbb{C},\mathbb{C})\ar[d]^{\cong}\\
\mathrm{HL}(C^*G,C(B))\otimes \mathrm{HL}(C(B),\mathbb{C})\ar[r]^{\hspace*{1.8cm}\circ} &\mathrm{HL}(C^*G,\mathbb{C}).
}$$
From this we deduce that the product $\Ch^{\mathrm{HL}}(\ind ([\sigma ]))\circ \Ch^{\mathrm{HL}}([P_{\sigma'}])$ becomes via the isomorphism $\mathrm{HL}(C^*G ,\mathbb{C})\simeq \mathrm{Hom}(R(G)\otimes \mathbb{C},\mathbb{C})$:
$$\Ch^{\mathrm{HL}}(\ind  ([\sigma ]))\circ \Ch^{\mathrm{HL}}([P_{\sigma'}])\cong\mathrm{Ind}^{G,M}_{a}([\sigma ]\otimes p^*[\sigma']).$$
Now, by Theorem \ref{thm:ind_c=ind_a}, see \cite{paradan2008index}, we know that $\mathrm{Ind}^{G,M}_a=\mathrm{Ind}^{G,M}_{coh}$. So
$$\Ch^{\mathrm{HL}}(\ind  ([\sigma ]))\circ \Ch^{\mathrm{HL}}([P_{\sigma'}])\cong\mathrm{Ind}^{G,M}_{coh}([\sigma ]\otimes p^*[\sigma']).$$
As $\mathrm{Ind}^{G,M}_{coh}([\sigma ]\otimes p^*[\sigma'])=\langle \mathrm{Ind}^{G,M|B}_{coh}([\sigma ]), PD(\Ch(P_{[\sigma']})\rangle $, we get by Poincaré duality that 
$$\Ch^{\mathrm{HL}}(\ind  ([\sigma ]))\simeq \mathrm{Ind}^{G,M|B}_{coh}([\sigma ]).$$
By the identification of Theorem \ref{thm:convergence:chern:indice}, we eventually get: $$\Ch( \ind  ([\sigma ]))= \mathrm{Ind}^{G,M|B}_{coh}([\sigma ]).$$
\end{proof}

\footnotesize
\bibliographystyle{plain}
\bibliography{Transversalement_elliptique}

\end{document}